\documentclass{amsart}

\usepackage{amsmath, amsthm, amssymb, amsfonts,enumerate}
\usepackage[foot]{amsaddr}
\usepackage[normalem]{ulem}
\usepackage[pdfusetitle, colorlinks=true]{hyperref}
\usepackage{color}
\usepackage{comment}

\theoremstyle{plain}

\newtheorem{theorem}{Theorem}[section]

\newtheorem{lemma}[theorem]{Lemma}
\newtheorem{lem}[theorem]{Lemma}

\newtheorem{corollary}[theorem]{Corollary}
\newtheorem{cor}[theorem]{Corollary}

\newtheorem{question}[theorem]{Question}

\newtheorem{proposition}[theorem]{Proposition}
\newtheorem{prop}[theorem]{Proposition}

\theoremstyle{definition}

\newtheorem{definition}[theorem]{Definition}

\theoremstyle{remark}
\newtheorem{remark}[theorem]{Remark}

\newtheorem{rem}[theorem]{Remark}

\DeclareSymbolFont{AMSb}{U}{msb}{m}{n}
\DeclareMathSymbol{\N}{\mathalpha}{AMSb}{"4E}
\DeclareMathSymbol{\R}{\mathalpha}{AMSb}{"52}
\DeclareMathSymbol{\Z}{\mathalpha}{AMSb}{"5A}
\DeclareMathSymbol{\D}{\mathalpha}{AMSb}{"44}
\DeclareMathSymbol{\s}{\mathalpha}{AMSb}{"53}

\newcommand{\MCP}{\mathrm{MCP}}
\newcommand{\CD}{\mathrm{CD}}
\newcommand{\RCD}{\mathrm{RCD}}
\newcommand{\CAT}{\mathrm{CAT}}
\newcommand{\CBA}{\mathrm{CBA}}

\newcommand{\DC}{\mathrm{DC}}

\newcommand{\BV}{\mathrm{BV}}

\newcommand{\eps}{\varepsilon}
\renewcommand{\phi}{\varphi}

\newcommand{\Xreg}{\mathcal{R}}

\DeclareMathOperator{\supp}{supp}

\DeclareMathOperator{\Ric}{Ric}
\DeclareMathOperator{\Lip}{Lip}

\DeclareMathOperator{\vol}{vol}

\renewcommand{\tilde}{\widetilde}

\newcommand{\Sph}{{S}} 
\newcommand{\cvertex}{\mathbf{0}}

\DeclareMathOperator{\diam}{diam}
\DeclareMathOperator{\rad}{rad}

\newcommand{\X}{X}

\newcommand{\ke}{{K}}

\newcommand{\uk}{\kappa}
\DeclareMathOperator{\Ch}{Ch}

\newcommand{\m}{m}

\definecolor{grey}{rgb}{0.800781, 0.800781, 0.800781}

\def\S2k{\mathbb{S}^2_\kappa}

\def\co{\colon\thinspace}
\def\SS{\mathbb S}
\def\const{\mathrm{const}}

\def\D{\mathbf D}
\def\greg{\mathcal R^g}
\def\CC{\mathcal C}
\def\TT{\mathcal T}

\def\reg{\mathcal R}
\def\greg{\mathcal R^g}

\DeclareMathOperator{\dims}{\dim_{split}}
\DeclareMathOperator{\dimg}{\dim_{geom}}

\DeclareMathOperator{\dimh}{\dim_{Haus}}
\DeclareMathOperator{\dimt}{\dim_{top}}

\newcommand{\bCAT}{\boldsymbol\partial}

\newcommand{\red}{\color{red}}
\newcommand{\blue}{\color{blue}}
\newcommand{\cyan}{\color{cyan}}

\DeclareMathOperator{\Intr}{Int}

\author{Vitali Kapovitch$^*$}
\address{$^*$University of Toronto}
\email{vtk@math.utoronto.ca}
\author{Martin Kell$^\dagger$}
\address{$^\dagger$Universit\"at T\"ubingen}
\email{martin.kell@math.uni-tuebingen.de}
\author{Christian Ketterer$^*$}
\email{ckettere@math.toronto.edu}

\title{On the structure of $\RCD$ spaces with upper curvature bounds}
\begin{document}
\begin{abstract} 
We develop a structure theory for $\RCD$ spaces with curvature bounded above in Alexandrov sense. In particular, we show  that any such space is a topological manifold with boundary whose interior is equal to the set of regular points. Further the set of regular points is a  smooth manifold and is geodesically convex. Around regular points there are $\DC$ coordinates and the distance is induced by a continuous $\BV$ Riemannian metric. 
\end{abstract}
\maketitle
\tableofcontents
\section{Introduction}
{A natural well-studied subclass of the class of  metric measure  spaces $(X,d,m)$ satisfying the $\RCD(K,N)$ condition with $K\in \R$ and $N\in (0,\infty)$ is given by $n$-dimensional Alexandrov spaces of curvature $\ge K/(n-1)$ with $n\le N$ equipped with the $n$-dimensional Hausdorff measure  $m=\mathcal H_n$. 
 In this paper we investigate another  natural  subclass  given by metric measure spaces $(X,d,m)$ such that
\begin{equation}\label{eq:rcd+cat}
\text{$(X,d,m)$ is $\RCD(K,N)$ and  $(X,d)$ is $\CAT(\kappa).$}
\end{equation}
This subclass is natural because it is stable under measured Gromov\textendash  Hausdorff convergence.
It is well known that there are Alexandrov spaces which are not $\CAT(\kappa)$ for any $\kappa$ even locally.
 On the other hand it was  shown in ~\cite{Kap-Ket-18} that there exist spaces satisfying \eqref{eq:rcd+cat} which have no lower curvature bounds in Alexandrov sense. Thus neither of these classes is contained in the other.  However it was shown in \cite{Kap-Ket-18}  that in the special case when $X$ satisfying \eqref{eq:rcd+cat} is \emph{non-collapsed} (i.e. when $N=n$ and $m=\mathcal H_n$) then $X$ is Alexandrov with
 curvature bounded from below by $K-(n-2)\kappa$.  It then follows by a theorem of   Berestovski\u{\i} and Nikolaev \cite{nikolaev} that the set of regular points is a $C^3$ manifold and the distance function derives from a $C^{1,\alpha}$-Riemannian metric.

The aim of this article is  to study the case of spaces satisfying   \eqref{eq:rcd+cat} in general. Note that it was shown in ~\cite{Kap-Ket-18} that  if $(X,d,m)$ is $\CD(K,N)$ and $\CAT(\kappa)$ then it's automatically infinitesimally Hilbertian and hence $\RCD(K,N)$. Thus $X$ satisfies \eqref{eq:rcd+cat} if and only if it is $\CD(K,N)$ and $\CAT(\kappa)$.

 In particular, we are interested in the structure of the regular set $\Xreg=\cup_{k\ge 0} \Xreg_k$ where  $\Xreg_k$ is the set of points that have a unique Gromov\textendash  Hausdorff tangent cone isometric  to $\mathbb R^k$. 
 
 It is known that for a general $\RCD(K,N)$ space  $(Y,d,m)$ with $N<\infty$ there exists unique $n\le N$ such $m(\reg_n)>0$ \cite{brusem}.
 The number  $n$ is  called the \emph{geometric dimension} of $Y$. We will denote it by $\dimg Y$.

 We also introduce a notion of the \emph{geometric boundary}\footnote{when there is no ambiguity we will often refer to the geometric boundary as just boundary} $\bCAT X$ of $X$ which is defined inductively on the geometric dimension and is analogous to the definition of the boundary of Alexandrov spaces and of non-collapsed $\RCD$ spaces as defined in~\cite{KapMon}.

Our main result is the following structure theorem

\begin{theorem}[Theorem \ref{th:reg-points}, Corollary~\ref{rcd+cat-imply-mnfld-bry}, Theorem~\ref{thm-DC-metric-structure}]\label{rcd-reg}

Suppose $X$ satisfies \eqref{eq:rcd+cat}. 

Then $X$ is a topological $n$-manifold with boundary where $n=\dimg X$ is the geometric dimension of $X$. Furthermore the following properties hold
\begin{enumerate}
\item $\Intr X$, the manifold  interior of $X$, is equal to $\reg_n$ and $\reg_k=\varnothing$ for $k\ne n$. In particular  $\Xreg=\Xreg_n$.
\item $\reg$  is geodesically convex.
\item Geodesics in $\reg$ are locally extendible.
\item $\reg$ has a structure of a $C^1$-manifold with a $\BV\cap C^0$-Riemannian metric which induces the original distance $d$.
\item The manifold  boundary $\partial X$ of $X$ is equal to the geometric boundary $\bCAT X$.
\end{enumerate}

\end{theorem}

In the special case when $\uk=1, K=N-1$ we obtain the following rigidity result.

\begin{theorem}[Sphere-Theorem (Corollary~\ref{rcd+cat-sphere-thm})]\label{intro:rcd+cat-sphere-thm}
Let $(X,d,m)$ be $\RCD(N-1,N)$ and $\CAT(1)$ { for some $N>1$}.

 If $\bCAT X\ne \varnothing$ then $X$ is homeomorphic to a closed disk of dimension $\le N$.

If $\bCAT X= \varnothing$ then $N$ is an integer and $X$ is metric measure isomorphic to $(\SS^{N}, const\cdot \mathcal H_N)$.
\end{theorem}

 By the structure theory for $\RCD(K,N)$ spaces it is known that for a general $\RCD(K,N)$ space  $(Y,d,m)$ with $N<\infty$ it holds that $m$ is absolutely continuous with respect to $\mathcal H_n$   where $n=\dimg X$ (for instance \cite{mondinonaber, kellmondino, brusem}).
We show
\begin{theorem}[Subsection \ref{sec:noncol}]\label{th:function-semiconcave}
The limit $
f(x)=\lim\limits_{r\rightarrow 0} \frac{\m(B_r(x))}{\omega_n r^n}
$ exists for all $x\in \reg$. 

Furthermore,  $f$ is semi-concave and  locally Lipschitz on $\Xreg$ and $m|_\reg=f\cdot \mathcal H_n$.

\end{theorem}

Finally, we consider the subclass of \emph{weakly non-collapsed} $\RCD$ spaces with upper curvature bounds. Weakly non-collapsed $\RCD(K,N)$ spaces  are $\RCD(K,N)$ spaces for which the geometric dimension $n$ is the maximal possible, i.e. it  is equal to $N$.  
It was conjectured by Gigli and DePhilippis in \cite{GP-noncol} that in this case up to a constant multiple  $\m=\mathcal H_n$ i.e. up to rescaling the measure by a constant weakly non-collapsed spaces are non-collapsed.  The conjecture was proved for spaces satisfying ~\eqref{eq:rcd+cat} in \cite{KaKe19}. In \cite{Han19} Han proved the conjecture when the underlying metric space is a smooth Riemannian manifold with boundary.


In this paper we give a much simpler proof of this conjecture for spaces satisfying ~\eqref{eq:rcd+cat} using somewhat similar ideas to those in  \cite{Han19}.
\begin{theorem}[Corollary \ref{cor:noncol}]\label{cat+cd-weak-noncol}
Suppose $X$ satisfies \eqref{eq:rcd+cat} and  is  weakly non-collapsed. Then $f=\const$ $\m$-almost everywhere on $X$.
\end{theorem}

 Some of the results of this article already appear in \cite{KaKe19}.

We remark that the previous results also apply if $(X,d,m)$ is $\RCD$ and $(X,d)$ has curvature bounded above by $\kappa$ ($\CBA(\kappa)$), i.e. it is $\CAT(\kappa)$ locally. In this case small  closed balls  are geodesically convex  and are  $\CAT(\kappa)$  and therefore $\RCD$. Hence, if we replace the condition $\CAT(\kappa)$ by $\CBA(\kappa)$,  the conclusions of the previous theorems are true for such balls, and in the case of Theorem \ref{rcd-reg}, Theorem \ref{th:function-semiconcave} and Theorem \ref{cat+cd-weak-noncol} the conclusion globalizes to the whole space.

Note however that the globalization theorem for the condition $\CBA$ holds only under additional topological constraints. More precisely, a metric space that satisfies $\CBA(\kappa)$ is $\CAT(\kappa)$ if and only if any loop of length less than  $2\pi_{\kappa}$ is contractible through loops of length less than $2\pi_{\kappa}$ where $\pi_\kappa\in (0,\infty]$ is the diameter of a simply connected space $\mathbb S_{\kappa}^2$ of constant curvature $\kappa$.
Moreover, unlike $\CAT(\kappa)$ the condition $\CBA(\kappa)$  is not stable under Gromov-Hausdorf convergence.

In particular, for fixed $N, K,\kappa$  the above discussion applies to closed convex domains in  weighted Riemannian manifolds $(M,g, m)$ which are $\RCD(K,N)$ and have sectional curvature bounded above by $\kappa$,  hence { recovering} the corresponding results in \cite{Han19}. 

In the special case when $m$ is smooth, i.e. $m= f\cdot d\vol_g$ for some smooth function $f$, this means that the theory applies to closed convex domains in 
$(M^n,g, f\cdot d\vol_g)$ with $n\le N$, $\sec\le \kappa$ and the Bakry-Emery Ricci curvature of $M$  satisfying $\Ric_{N,f}\ge K$. Moreover if such domains in addition satisfy the above globalization assumption on loops then their measured Gromov\textendash Hausdorff limits are $\CAT(\kappa)$ and $\RCD(K,N)$ and hence the theory applies to them as well.

While we were preparing this article we became aware of a recent preprint \cite{Honda19} by S. Honda. He confirms Gigli and DePhilippis' conjecture in the compact case without assuming an upper curvature bound, though his argument is again  more involved and based on the  $L^2$-embedding theory for $\RCD$ space via the heat kernel \cite{AHPT18}. This does imply Theorem~\ref{cat+cd-weak-noncol} for general, possibly noncompact $X$ satisfying \eqref{eq:rcd+cat} because condition  \eqref{eq:rcd+cat}  implies that all small closed balls in $X$ are convex and compact.

To close this introduction we remark again that the class of spaces satisfying ~\eqref{eq:rcd+cat} is natural because it is stable w.r.t. measured Gromov\textendash  Hausdorff convergence. 
As a subclass of the class of $\RCD$ spaces it can serve as a test case for proving general conjectures about $\RCD$ spaces. In particular in Section ~\ref{sec: t-cones-cont} we prove that for spaces  satisfying ~\eqref{eq:rcd+cat} same scale tangent cones are measured Gromov\textendash  Hausdorff continuous along interiors of geodesics. This is known to be true for limit geodesics in Ricci limits  ~\cite{coldingnaberI,Kap-Li} but is currently not known for general $\RCD(K,N)$ spaces.

It is also quite natural to relax the assumptions further and replace the $\RCD$ condition by the even weaker but still measured Gromov\textendash  Hausdorff stable measure contraction property $\MCP$ \cite{ohtamcp, stugeo2}. The authors plan to investigate this class in a separate paper.

The article is structured as follows. 

In Section \ref{sec: prelim} we recall definitions and results about $\RCD$ spaces with curvature bounded above and calculus in $\DC$ coordinates.

In Section \ref{sec: CD+CAT-str} we develop structure theory of $\RCD$+$\CAT$ spaces.

In Section~\ref{sec:boundary} we introduce and study the geometric boundary of $\RCD$+$\CAT$ spaces.

In Section  \ref{sec:dc} we study the  $\DC$ structure on $\reg$.

In Section \ref{sec: dens} we study the density function, first in the non-collapsed,  and then in the collapsed case.

In Section ~\ref{sec: t-cones-cont} we prove the continuity of same scale tangent cones along interiors of geodesics.

In the development of the structure theory of $\RCD$+$\CAT$ spaces in sections \ref{sec: CD+CAT-str}  and ~\ref{sec:boundary} only a few consequences of the $\RCD$ condition were used, chief of which are non-branching and the splitting theorem.
In Section~\ref{sec-class-c}  we try to find the fewest number of extra conditions one needs in addition to $\CAT$ to make much of the structure theory from sections \ref{sec: CD+CAT-str}  and ~\ref{sec:boundary} work.

\subsection{Acknowledgements}
The first author is funded by a Discovery grant from NSERC.  
The third author is funded by the Deutsche Forschungsgemeinschaft (DFG, German Research Foundation) -- Projektnummer 396662902.
We are grateful to Alexander Lytchak for a number of  helpful conversations.
\section{Preliminaries}\label{sec: prelim}
\subsection{Notations}
Given a metric space $X$ we will denote by  $C(X)$  the Euclidean cone over  $X$ and  by $S(X)$  spherical suspension over $X$. 
We use the parametrization $X\times[0,\infty)$ for $C(X)$ and $X\times[0,\pi]$ for $S(X)$ and use $\cvertex$ for the point $(x,0)$ 
in $C(X)$ and resp. $S(X)$.

For $p\in X$ we'll denote by $B_r(p)$ the ball of radius $r$ centered at $p$ and by $S_r(p)$ the sphere of radius $r$ around $p$.

 Given $p,q$ in a geodesic metric space $X$ we'll denote by $[x,y]$ a shortest geodesic between $p$ and $q$. 

The \emph{radius} of metric space $X$ is defined as $\rad X:=\inf r$  such that $B_r(p)=X$ for some $p\in X$.

\subsection{Curvature-dimension condition}
A \textit{metric measure space} is a triple $(X,d,\m)$ where $(X,d)$ is a complete and separable metric space and $\m$ is a locally finite measure.

$\mathcal{P}^2(X)$ denotes the set of Borel probability measures $\mu$ on $(X,d)$ such that $\int_Xd(x_0,x)^2d\mu(x)<\infty$ for some $x_0\in X$ equipped with the $L^2$-Wasserstein distance $W_2$. The subspace of $\m$-absolutely continuous probability measures in $\mathcal{P}^2(X)$ is denoted  $\mathcal{P}^2(X,\m)$.

The \textit{$N$-Renyi entropy} is
\begin{align*}
S_N(\cdot|\m):\mathcal{P}^2_b(X)\rightarrow (-\infty,0],\ \ S_N(\mu|\m)=\begin{cases}-\int \rho^{1-\frac{1}{N}}d\m& \ \mbox{ if $\mu=\rho\m$, and }\smallskip\\
0&\ \mbox{ otherwise}.
\end{cases}
\end{align*}
$S_N$ is lower semi-continuous, and $S_N(\mu)\geq - \m(\supp\mu)^{\frac{1}{N}}$ by Jensen's inequality.

For $\kappa\in \mathbb{R}$ we define 
\begin{align*}
\cos_{\kappa}(x)=\begin{cases}
 \cosh (\sqrt{|\kappa|}x) & \mbox{if } \kappa<0\\
1& \mbox{if } \kappa=0\\
\cos (\sqrt{\kappa}x) & \mbox{if } \kappa>0
                \end{cases}
                \quad \& \quad
   \sin_{\kappa}(x)=\begin{cases}
\frac{ \sinh (\sqrt{|\kappa|}x)}{\sqrt{|\kappa|}} & \mbox{if } \kappa<0\\
x& \mbox{if } \kappa=0\\
\frac{\sin (\sqrt{\kappa}x)}{\sqrt \kappa} & \mbox{if } \kappa>0.
                \end{cases}                 
                \end{align*}
Let $\pi_\kappa$ be the diameter of a simply connected space form $\S2k$ of constant curvature $\kappa$, i.e.
\[
\pi_\kappa= \begin{cases}
 \infty \ &\textrm{ if } \kappa\le 0\\
\frac{\pi}{\sqrt \kappa}\ &  \textrm{ if } \kappa> 0.

\end{cases}
\]
For $K\in \mathbb{R}$, $N\in (0,\infty)$ and $\theta\geq 0$ we define the \textit{distortion coefficient} as
\begin{align*}
t\in [0,1]\mapsto \sigma_{K,N}^{(t)}(\theta)=\begin{cases}
                                             \frac{\sin_{K/N}(t\theta)}{\sin_{K/N}(\theta)}\ &\mbox{ if } \theta\in [0,\pi_{K/N}),\\
                                             \infty\ & \ \mbox{otherwise}.
                                             \end{cases}
\end{align*}
Note that $\sigma_{K,N}^{(t)}(0)=t$.
For $K\in \mathbb{R}$, $N\in [1,\infty)$ and $\theta\geq 0$ the \textit{modified distortion coefficient} is
\begin{align*}
t\in [0,1]\mapsto \tau_{K,N}^{(t)}(\theta)=\begin{cases}
                                            \theta\cdot\infty \ & \mbox{ if }K>0\mbox{ and }N=1,\\
                                            t^{\frac{1}{N}}\left[\sigma_{K,N-1}^{(t)}(\theta)\right]^{1-\frac{1}{N}}\ & \mbox{ otherwise}.
                                           \end{cases}\end{align*}
\begin{definition}[\cite{stugeo2,lottvillani}]
We say $(X,d,\m)$ satisfies the \textit{curvature-dimension condition} $\CD(\ke,N)$ for $\ke\in \mathbb{R}$ and $N\in [1,\infty)$ if for every $\mu_0,\mu_1\in \mathcal{P}_b^2(X,\m)$ 
there exists an $L^2$-Wasserstein geodesic $(\mu_t)_{t\in [0,1]}$ and an optimal coupling $\pi$ between $\mu_0$ and $\mu_1$ such that 
$$
S_N(\mu_t|\m)\leq -\int \left[\tau_{K,N}^{(1-t)}(d(x,y))\rho_0(x)^{-\frac{1}{N}}+\tau_{K,N}^{(t)}(d(x,y))\rho_1(y)^{-\frac{1}{N}}\right]d\pi(x,y)
$$
where $\mu_i=\rho_id\m$, $i=0,1$.

\begin{remark}
If $(X,d,\m)$ is complete and satisfies the condition $\CD(\ke,N)$ for $N<\infty$, then $(\supp \m, d)$ is a geodesic space and $(\supp\m,  d,\m)$ is 
$\CD(\ke,N)$. 
\\
\\
In the following we always assume that $\supp\m=X$.
\end{remark}
%
%
\end{definition}
\subsection{Calculus on metric measure spaces}
For further details about the { properties of $\RCD$ spaces} we refer to
\cite{agslipschitz,agsheat,agsriemannian,giglistructure, gmsstability}.

Let $(X,d,m)$ be a metric measure space, and let $\Lip(X)$ be the space of Lipschitz functions. 
For $f\in \Lip(X)$ the local slope is
\begin{align*}
\mbox{Lip}(f)(x)=\limsup_{y\rightarrow x}\frac{|f(x)-f(y)|}{d(x,y)}, \ \ x\in X.
\end{align*}
If $f\in L^2(m)$, a function $g\in L^2(\m)$ is called \textit{relaxed gradient} if there exists sequence of Lipschitz functions $f_n$ which $L^2$-converges to $f$, and there exists $h$ such that 
$\mbox{Lip}f_n$ weakly converges to $h$ in $L^2(m)$ and $h\leq g$ $\m$-a.e.\ . A function $g\in L^2(\m)$ is called the \textit{minimal relaxed gradient} of $f$ and denoted by $|\nabla f|$ if it is a relaxed gradient and minimal w.r.t. the $L^2$-norm amongst all relaxed gradients.
The space of \textit{$L^2$-Sobolev functions} is then $$W^{1,2}(X):= D(\Ch^X):= \left\{ f\in L^2(\m): \int |\nabla f|^2 d\m<\infty\right\}.$$
%
$W^{1,2}(X)$ equipped with the norm 
$
\left\|f\right\|_{W^{1,2}(X)}^2=\left\|f\right\|^2_{L^2}+\left\||\nabla f|\right\|_{L^2}^2
$
is a Banach space.
If $W^{1,2}(X)$ is a Hilbert space, we say $(X,d,m)$ is \textit{infinitesimally Hilbertian.}

In this case one can define 
\begin{align}\label{rcdinnerproduct}
(f,g)\in W^{1,2}(X)^2\mapsto \langle \nabla f,\nabla g\rangle := \frac{1}{4}|\nabla (f+g)|^2-\frac{1}{4}|\nabla (f-g)|^2\in L^1(m).
\end{align}

Assuming $X$  is locally compact, if $U$ is an open subset of $X$, we say that $f\in W^{1,2}(X)$ is in the domain $D({\bf \Delta},U)$ of the \textit{measure-valued Laplace} ${\bf \Delta}$ on $U$ if there exists a signed Radon functional ${\bf \Delta}f$ on the set of all Lipschitz functions $g$ with bounded support in $U$ such that
\begin{align}\label{equ:integrationbyparts}
\int\langle \nabla g,\nabla f\rangle dm = -\int g d{\bf \Delta}f.
\end{align}
If $U=X$ and ${\bf \Delta}f= [{\bf\Delta} f]_{ac} \m$ with $[{\bf\Delta} f]_{ac}\in L^2(\m)$, we write $[{\bf\Delta} f]_{ac}=:\Delta f$ and $D({\bf \Delta}, X)=D_{L^2(\m)}(\Delta)$. 
$\mu_{ac}$ denotes the $\m$-absolutely continuous part in the Lebesgue decomposition of a Borel measure $\mu$.
\begin{definition}\label{def:rcd}
A metric measure space $(X,d,m)$ satisfies the \emph{Riemannian curvature-dimension condition} $\RCD(\ke,N)$ for $\ke\in \mathbb{R}$ and {$N\in [1,\infty)$} if it satisfies a curvature-dimension
conditions $\CD(\ke,N)$ and is infinitesimally Hilbertian.
\end{definition}
{Let $(X,d,\m)$ be an $RCD(K,N)$ space for some $K\in \R$ and $N\in (0,\infty)$.}
{The set of $k$-regular points $\Xreg_k$
is the collection of all points $p\in X$  such that \emph{every} mGH-tangent cone  is isomorphic to 
$(\mathbb R^k, d_{\mathbb R^k}, c_k \mathcal {H}^k)$ for some positive constant $c_k$. The union $\Xreg=\cup_{k\ge 0} \Xreg_k$ is the set of all regular points.

By \cite{brusem} one has that there exists $n\in \mathbb{N}$ (called the \emph{geometric dimension} of $X$) such 
that the set of $n$-regular points has full $\m$-measure. 
}

\subsection{Spaces with upper curvature bounds}

We will assume familiarity with the notion of $\CAT(\kappa)$ spaces. We refer to ~\cite{BBI, BH99} or ~\cite{Kap-Ket-18} for the basics of the theory.
\begin{definition}
Given a point $p$ in a $\CAT(\kappa)$ space $X$ we say that two unit speed geodesics starting at $p$ define the same direction if the angle between them is zero. This is an equivalence relation by the triangle inequality for angles and the angle induces a metric on the set $S_p^g(X)$ of equivalence classes. The metric completion  $\Sigma_p^gX$ of $S_p^gX$ is called the \emph{space of geodesic directions} at $p$.
The Euclidean cone $C(\Sigma_p^gX)$ is called the \emph{geodesic tangent cone} at $p$ and will be denoted by $T^g_pX$.
\end{definition}
The following theorem is due to Nikolaev~\cite[Theorem 3.19]{BH99}:
\begin{theorem}\label{geod-tangent-cone}
$T_p^gX$ is $\CAT(0)$ and $\Sigma_p^gX$ is  $\CAT(1)$. 
\end{theorem}
Note that this theorem in particular implies that $T_p^gX$ is a geodesic metric space which is not obvious from the definition.
More precisely, it means  that each path component of $\Sigma_p^gX$ is $\CAT(1)$ (and hence geodesic) and the distance between points in different components is $\pi$. Note however, that $\Sigma_p^gX$ itself need not to be path connected. 
\smallskip

We use the following terminology: a point $v\in\Sigma$
in a $\CAT(1)$ space has an \emph{opposite} $-v$ if  $d(v,-v)\ge \pi$.
This is easily seen to be equivalent to the statement that 
\begin{align*}
t\mapsto\begin{cases}
(v,t) & t\ge0\\
(-v,-t) & t\le0
\end{cases}
\end{align*}
 is a geodesic line in the Euclidean cone $C(\Sigma)$ over $\Sigma$.

Similarly, if $\gamma:[0,1]\to X$ is a geodesic in a $\CAT(\kappa)$ space
and $s\in(0,1)$ then $\dot{\gamma}(s)$ denotes the point in $\Sigma_{\gamma(s)}X$
corresponding to the direction of $s'\mapsto\gamma(s')$ at $s'=s$.
It is easy to verify that $\dot{\gamma}(s)$ has an opposite which
we denote by $-\dot{\gamma}(s')$. 

\subsection{$\BV$ functions and $\DC$ calculus}\label{subsection:BV}
Recall that a function $g: V\subset \mathbb R^n\rightarrow \mathbb R$ of bounded variation ($\BV$) admits a derivative in the distributional sense \cite[Theorem 5.1]{Gar-Evans} that is a signed vector-valued Radon measure 
$[Dg]=(\frac{\partial g}{\partial x_1}, \dots, \frac{\partial g}{\partial x_n})=[Dg]_{ac}+[Dg]_s$. Moreover, if $g$ is $\BV$, then it is $L^1$ differentiable \cite[Theorem 6.1]{Gar-Evans} a.e. with $L^1$-derivative $[Dg]_{ac}$, and approximately differentiable a.e. \cite[Theorem 6.4]{Gar-Evans} with approximate derivative $D^{ap}g=(\frac{\partial^{ap} g}{\partial x_1}, \dots, \frac{\partial^{ap} g}{\partial x_n})$ that coincides almost everywhere with $[Dg]_{ac}$. The set of $\BV$ functions $\BV(V)$ on $V$ is closed under addition and multiplication \cite[Section 4]{Per-DC}. 
We'll call $\BV$ functions $\BV_0$ if they are continuous. 
{
\begin{remark}
In ~\cite{Per-DC} and ~\cite{ambrosiobertrand}  $\BV$ functions are called  $\BV_0$ if they are continuous away from an $\mathcal H_{n-1}$-negligible set.  However, for the purposes of the present paper it will be more convenient to work with the more restrictive definition above.
\end{remark}
}

For  $f,g\in \BV_0(V)$ we have 
\begin{align}\label{equ:leibniz}
\frac{\partial (f g)}{\partial x_i}= \frac{\partial f}{\partial x_i} g + f \frac{\partial g}{\partial x_j} 
\end{align}
as signed Radon measures \cite[Section 4, Lemma]{Per-DC}. By taking the $\mathcal{L}^n$-absolutely continuous
part of this equality it follows that  \eqref{equ:leibniz} also holds a.e. in the sense of approximate derivatives. In fact, it  holds at \emph{all} points of approximate differentiability of $f$ and $g$. 
This easily follows by a minor variation of the standard  proof  that $d(fg)=fdg+gdf$ for differentiable functions.
\medskip

A function $f: V\subset \mathbb{R}^n\rightarrow \mathbb{R}$ is called a $\DC$ function if in a small neighborhood of each point $x\in V$ one can write $f$ as a difference of two semi-convex functions. The set of $\DC$ functions on $V$ is denoted by $\DC(V)$ and contains the class $C^{1,1}(V)$. The set $\DC(V)$ is closed under addition and multiplication.
The first partial derivatives $\frac{\partial f}{\partial x_i}$ of a $\DC$ function $f: V\rightarrow \mathbb R$ are $\BV$, and hence the second partial derivatives $\frac{\partial}{\partial x_j}\frac{\partial f}{\partial x_i}$ exist as signed Radon measure that satisfy $$\frac{\partial}{\partial x_i}\frac{\partial f}{\partial x_j}=\frac{\partial}{\partial x_j}\frac{\partial f}{\partial x_i}$$
\cite[Theorem 6.8]{Gar-Evans}, and hence 
\begin{equation}\label{2n-der-commute}
\frac{\partial^{ap}}{\partial x_i}\frac{\partial f}{\partial x_j}=\frac{\partial^{ap}}{\partial x_j}\frac{\partial f}{\partial x_i}\quad \text{  a.e. on $V$.}
\end{equation}
A map $F: V\rightarrow \mathbb{R}^l$, $l\in\mathbb{N}$, is called a $\DC$ map if each coordinate function $F_i$ is $\DC$. The composition of two $\DC$--maps is again $\DC$.
A function $f$ on $V$ is   called $\DC_0$ if it's $\DC$ and $C^1$.
\medskip

Let $(X,d)$ be a geodesic metric space. A function $f:X\rightarrow \mathbb R$ is called a $\DC$ function if it can be locally represented as the difference of two Lipschitz  semi-convex functions. A map $F:Z\rightarrow Y$ between metric spaces $Z$ and $Y$ that is locally Lipschitz is called a $\DC$ map if for each 
$\DC$ function $f$ that is defined on an open set $U\subset Y$ the composition $f\circ F$ is $\DC$ on $F^{-1}(U)$. In particular, a map $F: Z\rightarrow \mathbb R^l$ is $\DC$
if and only if its coordinates are $\DC$. If $F$ is a bi-Lipschitz homeomorphism and its
inverse is $\DC$, we say that $F$ is a $\DC$-isomorphism.

 \section{Structure  theory of $\RCD$+$\CAT$ spaces}\label{sec: CD+CAT-str}

In this section we study the following class of metric  measure spaces
\begin{equation}\label{eq:cd+cat}
\begin{gathered}
\mbox{$(X,d,m)$ is $\CAT(\uk)$ and satisfies the condition $\RCD(K,N)$ for some $1\le N<\infty$, $K,\uk<\infty$.}
\end{gathered}
\end{equation}

The following result was proved in ~\cite{Kap-Ket-18}

\begin{theorem}[\cite{Kap-Ket-18}]\label{CD+CAT implies RCD}
Let $(X,d,\m)$ satisfy $\CD(K,N)$ and  $\CAT(\uk)$ for $1\le N<\infty$, $K,\uk\in \mathbb{R}$. Then $X$ is infinitesimally Hilbertian.
In particular, $(X,d,\m)$ satisfies $\RCD(K,N)$.
\end{theorem}
\begin{remark}
It was shown in \cite{Kap-Ket-18} that the above theorem also holds if the $\CD(K,N)$ assumption is replaced by $\CD^*(K,N)$ or $\CD^e(K,N)$ conditions (see \cite{Kap-Ket-18} for the definitions).
Moreover, in a recent paper~\cite{DGPS18} Di Marino, Gigli,  Pasqualetto and Soultanis show that a $\CAT(\kappa)$ space with \emph{any} Radon measure  is infinitesimally Hilbertian. For these reasons \eqref{eq:cd+cat} is equivalent to assuming that $X$ is $\CAT(\uk)$ and satisfies one of the assumptions  $\CD(K,N), CD^*(K,N)$ or $\CD^e(K,N)$ with  $1\le N<\infty$, $K,\uk<\infty$.

\end{remark}
The following key property of spaces satisfying  \eqref{eq:cd+cat} was also established in  \cite{Kap-Ket-18}
\begin{proposition}[\cite{Kap-Ket-18}]\label{prop:nonbra}
Let $X$ satisfy \eqref{eq:cd+cat}. Then $X$ is non-branching.
\end{proposition}

Recall that a metric space $X$ is called $C$-\emph{doubling} with respect to a non-decreasing function $C\co (0,\infty)\to (1,\infty)$ if  for any $p\in X$ and any $r>0$ the ball $B_r(p)$ can be covered by $C(r)$ balls of radius $r/2$.
The  doubling condition implies that for any $p\in X$ and any $r_i\to 0$ the sequence $(\frac{1}{r_i}X,p)$ is precompact in the  pointed Gromov\textendash  Hausdorff topology. Therefore one can define tangent cones $T_pX$ at $p$ as limits of such subsequences. Obviously, any tangent cone $T_pX$ is $\CAT(0)$.  We will frequently make use of the following general lemma.
\begin{lemma}\label{geom-vs-blowup-t-cones}
Let $X$ be $\CAT(\uk)$ and  $C$-doubling for some non-decreasing  $C\co (0,\infty)\to (1,\infty)$ and let $p\in X$. Then
\begin{enumerate}[(i)]
\item\label{geom-tcone-embeds} For any tangent cone $T_pX$  the geodesic tangent cone $T_p^gX$ isometrically embeds into $T_pX$ as a convex closed subcone. In particular $\Sigma_p^gX$ is compact.
\item\label{t-cone-emb-onto} If there exists $\eps>0$ such that every geodesic starting at $p$ extends to length $\eps$ then the embedding from part \eqref{geom-tcone-embeds} is onto. In particular $T_pX$ is unique and is isometric to $T_p^gX$.
\end{enumerate}
\begin{proof}
Let $T_pX$ be a tangent cone at $p$. The doubling condition passes to the limit and becomes globally $C(1)$-doubling, i.e. any ball of any radius $r>0$ in $T_pX$ can be covered by $C(1)$ balls of radius $r/2$. This implies that $T_pX$ is proper, i.e. all closed balls in $T_pX$ are compact.   Let $\eps<1/100$ and let  $v_1,\ldots v_k\in \Sigma_p^gX$ be a finite $\eps$-separated net given by geodesic directions. Let $\alpha_{ij}=\angle v_i v_j$. Let $\gamma_i(t), i=1,\ldots, k$, be  unit speed geodesic with $\gamma_i(0)=p,\gamma_i'(0)=v_i$. Then by the definition of angles we have that $d(\gamma_i(t), \gamma_j(s))=\sqrt{t^2+s^2-2st\cos \alpha_{ij}}+o(r)$ for $s,t\le r$. This immediately implies that the cone $C(\{v_1,\ldots, v_k\})$ isometrically embeds into $T_pX$ as a subcone. Furthermore, the images of $v_1,\ldots, v_k$ are $\eps/2$-separated in $T_pX$. Since $T_pX$ is $C(1)$-doubling it holds that  $k\le n=n(C(1),\eps)$. Since this holds for all small $\eps$ we get that $ \Sigma_p^gX$ is compact. Now a diagonal Arzela-Ascoli argument gives that there is a distance preserving embedding $f\co T_p^gX\to T_pX$. Since both spaces are geodesic and geodesics in $CAT(0)$ spaces are unique this implies that the image $f(T_p^gX)$ is a convex subset of $T_pX$. Since $f$ is continuous and $T_p^gX$ is proper we can also conclude that $f(T_p^gX)$ is closed.  This proves part \eqref{geom-tcone-embeds}.

Now suppose that all geodesics starting at $p$ extend to uniform distance $\eps>0$. Let $0<R<\min\{\eps, 1,\pi_\kappa/100\}$. Let $\delta>0$ and choose a finite $\delta\cdot R$ net in $S_R(p)$ given by $\gamma_i(Rv_i), i=1,\ldots, k$ for some $v_1,\ldots,v_k\in \Sigma_p^gX$ and unit speed geodesics $\gamma_1,\ldots, \gamma_k$ with $\gamma_i(0)=p,\gamma_i'(0)=v_i$. Then the $\CAT(\kappa)$-condition implies that for any $0<r\le R$  the set $\cup_i \gamma_i([0,r])$ is $\delta\cdot r$ dense in  $B_r(p)$. This implies that for the embedding $f\co T_p^gX\to T_pX$ constructed in part \eqref{geom-tcone-embeds} the image of the unit ball around the vertex in $T_p^gX$ is $\delta$-dense in the unit ball around the vertex in $T_pX$. Since this holds for any $\delta>0$ and the image of $f$ is closed we get that $f$ is onto. This proves  \eqref{t-cone-emb-onto}.
\end{proof}
\end{lemma}
The above Lemma obviously applies to spaces satisfying  \eqref{eq:cd+cat}.
We currently don't know if  for such spaces the embeddings $T_p^gX\subset T_pX$  constructed in part  \eqref{geom-tcone-embeds} of the lemma are always onto.
\begin{remark}\label{geod-cone-rcd}


Recall that in  $\CAT(0)$ spaces distance functions to convex sets are convex and therefore  an $\eps$-neighbourhood of a convex set is convex. Therefore, even if $T_p^gX\subset T_pX$ has  measure zero it still inherits the structure of an $\RCD(0,N)$ space as follows. Consider $Y_\eps=U_\eps(T_p^gX)\subset T_pX$ and equip it with the renormalized measure $m_\infty^\eps={m_\infty(B_1(o)\cap Y_\eps)}^{-1}m_\infty|_{Y_\eps}$. Then $(Y_\eps,d_\infty,m_\infty^\eps,o)$ is $\RCD(0,N)$ and as $\eps\to 0$ it pmGH-subconverges to $(T_p^gX,d_\infty, m_\infty^g,o)$ for some (possibly non-unique) limit measure $m_\infty^g$ and this space is $\RCD(0,N)$.  Note however, that even though $T_p^gX$ is a metric cone by construction, it's not clear if $(T_p^gX,d_\infty, m_\infty^g,o)$ is always a volume cone. Therefore we can not conclude that $\Sigma_p^g$ has any natural measure that turns it into and $\RCD$ space. Nevertheless the splitting theorem guarantees that for any $v\in \Sigma_p^g$ it holds that $T_vT_p^gX\cong \R\times T_v\Sigma_p^g X$ and therefore  $T_v^g\Sigma_p^g X$ does inherit a natural structure of an $\RCD(0,N-1)$ space.
\end{remark}

The following can be obtained by adjusting the proof of \cite[Theorem A]{Kramer11}
(see Footnote $5$ of \cite[Section 3]{Kramer11}). 
\begin{lem}\label{lem-kramer}
Let $(X,d)$ be a non-branching $\CAT(\kappa)$ space and $\gamma\co [0,1]\to X$
be a geodesic. Then for all balls $\bar{B}_{r}(\gamma_{t})$, $t\in(0,1)$,
$r<\frac{\pi_{\kappa}}{2}$ with $\gamma_{0},\gamma_{1}\notin\bar{B}_{r}(\gamma_{t})$, for any $s$ such that $\gamma(s)\in B_{r}(\gamma(t))$
the space $\Sigma^g_{\gamma(s)}X\backslash\{\pm\dot{\gamma}(s)\}$
is homotopy equivalent to $\bar{B}_{r}(\gamma_{t})\backslash\gamma((0,1))$.
In particular, $\Sigma^g_{\gamma(s)}X\backslash\{\pm\dot{\gamma}(s)\}$,
$s\in(0,1)$, are homotopy equivalent. 
\end{lem}
\begin{proof}
Since $r<\frac{\pi_{\kappa}}{2}$, all geodesics in $\bar{B}_{r}(\gamma(t))$
are unique. As in \cite{Kramer11} there is a natural "log" map $\rho_{s}:\bar{B}_{r}(\gamma(t))\backslash\{\gamma(s)\}\to\Sigma^g_{\gamma(s)}X$
induced by the angle metric between geodesics starting at $\gamma(s)$
and ending in a point $\bar{B}_{r}(\gamma(t))$. By \cite[Theorem A]{Kramer11}
this map is a homotopy equivalence.  Moreover the proof of \cite[Theorem A]{Kramer11} shows that for any { open}
$U_{s}\subset\Sigma_{\gamma(s)}X$ this map is a homotopy
equivalence between $\rho_{s}^{-1}(U_{s})$ and $U_{s}$, see \cite[Section 3, Footnote 5]{Kramer11}. 

Since $(X,d)$ is non-branching it holds that whenever $\xi:[0,1]\rightarrow \bar{B}_{r}(\gamma(t))$ is a geodesic with $\xi(0)=\gamma(s)$ and
$\dot{\xi}(0)=\pm\dot{\gamma}(s)$ then $\xi([0,1])\subset\gamma((0,1))$.
However, this implies that 
\[ 
\rho_{s}^{-1}(\Sigma^g_{\gamma(s)}X\backslash\{\pm\dot{\gamma}(s)\})=\bar{B}_{r}(\gamma(t))\backslash\gamma((0,1)).
\]
Since $\Sigma^g_{\gamma(s)}X\backslash\{\pm\dot{\gamma}(s)\}$ is open
in $\Sigma^g_{\gamma(s)}X$ the claim is proved.
\end{proof}

\begin{lem}\label{lem-geod-noncotr-2}
Assume $\Sigma$ is a spherical suspension over a $\CAT(1)$ space $Y$ and
denote the vertex points of $\Sigma$ by $\pm v\in\Sigma$. Then either
for all $w\in\Sigma$ there is an opposite direction $-w\in\Sigma$
or there is a $w\in\Sigma$ such that any geodesic between $w$ and
$\pm v$ cannot be extended beyond $\pm v$. In the latter case, both
spaces $\Sigma$ and $\Sigma\backslash\{\pm v\}$ are contractible.
\end{lem}
\begin{proof}
It suffices to assume $w\in \Sigma\backslash\{\pm v\}$. We may parametrize
points in $\Sigma\backslash\{\pm v\}$ by $Y\times(0,\pi)$. Assume $w=(z,s)$
for $s\in(0,\pi)$ and $z\in Y$. Let $\gamma:[0,\pi]\to\Sigma$
be a geodesic between $\tilde{v}$ and $-\tilde{v}$ with $\gamma(s)=(z,s)$
where $\tilde{v}\in\{\pm v\}$. If $\gamma$ can be extended to a
local geodesic $\gamma:[0,\pi+\epsilon] \rightarrow \Sigma$ beyond $-\tilde{v}$ then
there is a point $z'\in Y$ such that $\gamma(\pi+\epsilon)=(z',\pi-\epsilon)$.
Since local geodesics of length $\le \pi$ in $\CAT(1)$ spaces are geodesics  we see that $\gamma\big|_{[\epsilon,\pi+\epsilon]}$
is a minimal geodesic implying 
\[
d((z,\epsilon),(z',\pi-\epsilon)=\pi
\]
and thus $d(z,z')=\pi$. However, this shows that $(z',\pi-s)$ is
an opposite to $w$. 

If $w$ does not have an opposite then the open ball of radius $\pi$ around $w$
contains all points. However, uniqueness of those geodesics implies
that that the geodesic contraction $(\Phi)_{t\in[0,1]}$ towards $w$
induces a contraction of $\Sigma$ to $w$. To see that $\Sigma\backslash\{\pm v\}$
is also contractible observe that $\Phi_{t}(\Sigma)\subset\Sigma\backslash\{\pm v\}$
for $t\in[0,1)$ by choice of $w$, i.e. $\Phi_{t}:\Sigma\backslash\{\pm v\}\to\Sigma\backslash\{\pm v\}$ for any $t<1$.
\end{proof}

\begin{lemma}\label{cat-susp}
Let $\Sigma$ be a $\CAT(1)$ space. Suppose there are
 points  $v, -v\in \Sigma$ such that for any $x\in\Sigma$ it holds that $d(v,x)+d(-v,x)=\pi$.

Then $X$ is a spherical suspension over the $\CAT(1)$ space $Y=\{y\in \Sigma:\quad d(y,v)=d(y,-v)\}$ with vertices $v,-v$.
\end{lemma}
\begin{proof}
This is an immediate corollary of the "Lune Lemma" of Ballmann\textendash Brin \cite[Lemma 2.5]{Ballmann-Brin}.
\end{proof}

Next we isolate  the necessary properties of spaces satisfying \eqref{eq:cd+cat} that will be needed for developing their structure theory.
Note that if $(X,d,m)$ satisfies  \eqref{eq:cd+cat} then for an appropriately chosen large $\lambda$ the space $(X,\lambda d,m)$ satisfies $\RCD(-1,N)$ and $\CAT(1)$.
Therefore for the purposes of a structure theory we can always assume that $K=-1$ and $\kappa=1$ in \eqref{eq:cd+cat}.

Let $\mathcal C$ be a class of connected $\CAT(1)$ spaces  satisfying the following properties

\begin{enumerate}[(i)]
\item \label{c-closed-limits}$\mathcal C$ is closed under pointed GH-limits.
\item  \label{c-non-branch} Every $X\in \mathcal C$ is non-branching.
\item There is a non-decreasing function $C\co (0,\infty)\to (1,\infty)$ such that every $X\in \mathcal C$ is $C$-doubling.
\item \label{c-rescale} If $ X\in \CC$  then $\lambda X$ is also in $\CC$ for any $\lambda\ge 1$.
\item \label{c-convex}  If $X\in \CC$ and $p\in X$ then $T_p^gX\in \CC$ as well. 
\item \label{c-cone}If $C(\Sigma)\in \CC$ then $\Sigma\in \CC$ unless $C(\Sigma)\cong \R$.\footnote{$C(\Sigma)\cong \R$ is excluded since in this case $\Sigma $ is not connected.}
\item \label{c-split} If $C(\Sigma)\in \CC$ and if $v,-v\in  \Sigma$ are opposite then $\Sigma$ is a spherical suspension with vertices $v,-v$.
\end{enumerate}

By the previous discussion  an example of such class is given by the class consisting of spaces satisfying  \eqref{eq:cd+cat} with $K=-1,\kappa=1$ and of geodesic spaces of directions to points in such spaces.

Next we investigate geometric and topological properties of any class $\CC$ satisfying the above conditions.

The uniform doubling condition ensures that  $\CC$ is precompact in the pointed GH-topology which in conjunction with \eqref{c-rescale} means that we can talk about   (possibly non-unique) tangent cones at points of $X$ and all these tangent cones belong to $\CC$ as well. The doubling condition also implies that any $X\in \CC$ is proper (i.e. all closed balls are compact) and there is $N\in \N$ such that all $X\in \CC$ have Hausdorff dimension at most $N$.
By Lemma~\ref{geom-vs-blowup-t-cones} we have that for any $p\in X$ the geodesic tangent cone $T_p^gX$ embeds as a closed convex subset into any tangent cone $T_pX$. Since by \eqref{c-convex} $T_p^gX\in \CC$ condition  \eqref{c-cone} implies that the geodesic space of directions $\Sigma_p^gX\in \CC$ as well. 
Let us note here that this property makes the class $\CC$ more convenient to work with than the class of spaces satisfying \eqref{eq:cd+cat} which is not known to be closed under taking geodesic spaces of directions.

If $C(\Sigma)\in \CC$ then $\diam  \Sigma\le \pi$ since $C(\Sigma)$ is non-branching and moreover if $v\in \Sigma$ has an opposite then this opposite is unique.

Condition \eqref{c-split} further implies that in this case $C(\Sigma)$  satisfies the splitting theorem. This in particular applies to $T_p^g X$  for any $X\in \CC$ and any $p\in X$.

In analogy with $\RCD$ spaces given $X\in \CC$ and $m\in \N$ we say that a point $p\in X$ is $m$-regular if \emph{every} tangent cone $T_pX$ is isomorphic to $\R^m$.

Similarly we say that $p$ is geodesically $m$-regular if $T_p^gX\cong \R^m$. We set $\reg_m$ ( $\greg_m$ ) to be the set of all (geodesically) m-regular points and $\reg=\cup_m \reg_m$ ( $\greg=\cup_m \greg_m$) is the set of all (geodesically) regular points.  Note that since every $T_pX\in\CC$ we have that $\reg_m=\varnothing$ for $m>N$.

For $X\in \CC$ we  will call a point $p\in X$ \emph{inner} if every
geodesic $\gamma$ ending at $p$ can be locally extended beyond $p$.

Since all spaces in $\CC$ are of {finite Hausdorff dimension}  repeated application of \eqref{c-split} gives
\begin{lemma}\label{cor-cd-cut-susp}
Suppose $C(\Sigma)\in \CC$ and every point in $\Sigma$ has an opposite. Then $\Sigma\cong \SS^k$ for some $k\le N$.
\end{lemma}
This immediately yields
\begin{cor}\label{lem-inner-tangent}
Let $X\in \CC$. Suppose $p\in X$ is an inner point. Then  $p\in \greg_k$ for some $k\le N$. 
\end{cor}

\begin{prop}\label{reg=reg-g}
Let $X\in\CC$. Then $\Xreg_k \subset \greg_k$ for any $k\ge 0$.

\end{prop} 

\begin{proof}
Throughout the proof we'll denote by $\kappa(\delta)$ any function $\kappa\co (0,\infty)\to (0,\infty)$ such that $\kappa(\delta)\to 0$ as $\delta\to 0$.
Let $p\in \Xreg_k $ so that every tangent cone at $p$ is isometric to $\R^k$. This is equivalent to saying that $(\frac{1}{r} X, p)\overset{pGH}{\longrightarrow} (\R^k, 0)$ as $r\to 0$. We need to show that $T_p^gX\cong \R^k$.

It is enough to show that $H_{k-1}(B_r(p)\backslash \{p\})\ne 0$ for all small $r$. By ~\cite[Theorem A]{Kramer11}  this will imply that $\Sigma_p^g$ is not contractible, which by ~\cite[Theorem 1.5]{Lyt-Schr07} implies that all geodesics starting at $p$ are extendible to uniform distance $\eps>0$ 
which by  Lemma~\ref{geom-vs-blowup-t-cones} implies that the tangent cone  $T_pX$ is unique and isometric to $T_p^gX$.

Given $0<r<R$ we'll denote by $A_{r,R}(p)$ the {closed} annulus $\{r\le |xp|\le R\}$.

By assumption we have  that $(\frac{1}{r}X,p)\to (\R^k,0)$ as $r\to 0$. Let $r_i=1/2^i$. Let $f_i\co (\frac{1}{r_i}B_{r_i}(p), p)\to (B_1(0),0)$ be $\eps_i$ GH-approximations with $\eps_i\to 0$. Let $g_i$ be the "inverse" GH-approximations with $|id-f_i\circ g_i|\le \eps_i$ and $|id-g_i\circ f_i|\le \eps_i$.

Since both $\frac{1}{r_i}X$ and $\R^k$ are $\CAT(1/100)$ for large $i$, by a standard center of mass construction (e.g. by Kleiner \cite[{ Section 4}]{kleiner}) $f_i$ can be $\delta_i$-approximated by continuous maps with $\delta_i\to 0$. We will therefore assume that $f_i$ and $g_i$ are continuous to begin with. Note that the fundamental class of the sphere $S_{3/4}(0)$ is the generator of $H_{k-1}(A_{1/2,1}(0))\cong \Z$. Let $[c_i]=[g_{i}(S_{3/4}(0))]$ be its image in $H_{k-1}(B_{r_i}(p)\backslash\{p\})$. We claim that ${H_{k-1}(B_{r_i}(p)\backslash\{p\})\ni} [c_i]\ne 0$ for all large $i$ provided $\eps_i$ is small enough.

Suppose not and for some large $i$ we have that $[c_i]=0\in H_{k-1}(B_{r_i}(p)\backslash\{p\})$. Then $c_i=\partial w$ for some $k$-chain $w$ in $B_{r_i}(p)\backslash \left\{p\right\}$. Since the support of $w$ is compact in $B_{r_i}(p)\backslash \left\{p\right\}$, this implies that $[c_i]=0$ in the homology of some annulus $A_{\delta,r_i}(p)$ with $0<\delta<r_i$.  Applying radial geodesic contraction $\Phi_t$ to both $c_i$ and $w$ it follows that $(F_{\delta})_*(c_i)=0$ in $H_{n-1}(S_\delta(p))$ where $F_\delta$ is the nearest point projection onto $S_\delta(p)$ (it's Lipschitz and in particular continuous since $X$ is $\CAT(\kappa)$). 
Thus for all  $j$ such that $r_j<\delta$ it holds that $[z_j]=\Phi_{1/2^{j-i}}(c_i)=0$ in $H_{k-1}(A_{r_{j+1},r_j}(p))$.

We will show by induction on $j\ge i$ that $[z_j]\ne 0$ which will give a contradiction. In fact we claim that $[z_j]=\pm [c_j]$ in $H_{k-1}(A_{r_{j+1},r_j}(p))$ for all $j\ge i$. 

 This will give a contradiction when $j$ is large enough.

We only need to do the induction step from $j$ to $j+1$. Note that since $f_j\co (\frac{1}{r_i}B_{r_j}(p), p)\to (B_1(0),0)$ is an $\eps_j$-GH-approximation, it follows that the image of any radial geodesic $[px]$ in $\frac{1}{r_i}B_{r_j}(p)$ is $\kappa(\eps_j)$-close to the radial geodesic $[0f_j(x)]$ in $B_1(0)$. (The same is true in $\frac{1}{r_i}B_{r_j}(p)$ by the $\CAT(\kappa)$-condition). Therefore $f_j$ almost commutes with $\Phi_{1/2}$. That is $f_j(\Phi_{1/2}(x))$ is $\kappa(\eps_j)$-close to $\frac{1}{2}f_j(x)$ for $x\in B_{r_j}(p)$. The same holds for $g_i$. That is $g_j(\frac 1 2 y)$ is $\kappa(\eps_j)$-close to $\Phi_{1/2}(g_j(y))$ in $\frac{1}{r_i}B_{r_j}(p)$.

 Therefore if we rescale $B_{r_j}(p)$ and $B_1(0)$ by 2 (recall that $r_{j+1}=r_j/2$) it holds that $f_j(\Phi_{1/2}(x))$ is close to $f_j(x)$ in $2B_{1/2}(0)\cong B_1(0)$. Since close maps are homotopic via straight line homotopy it follows that $f_j(z_{j+1})$ is homologous to  $[S_{3/4}(0)] $ in $H_{k-1}(2A_{1/4,1/2}(0))\cong H_{k-1}(A_{1/2,1}(0))$. Note that the rescaled G-H approximation $f_j\co \frac{2}{r_j}B_{r_j/2}(p)\to 2B_{1/2}(0)\cong B_1(0)$ might be different from $f_{j+1}$ but they are $\kappa(\eps_j)$-close modulo post composition with an element of $O(k)$. An element of $O(k)$ maps $[S_{3/4}(0)]$ to $\pm [S_{3/4}(0)]$. Therefore $[z_{j+1}]=\pm [c_{j+1}]$ and the induction step is proved.

Thus $[z_i]\ne 0$ in $H_{k-1}(B_{r_i}(p)\backslash\{p\})$ which as was explained at the beginning implies the proposition.
\end{proof}

\begin{proposition}\label{inner-iff-non-contr-iff-non-contr-minus-geod} 
Let $X\in \CC$ and assume $\gamma\co (0,1)\to X$
is a non-trivial local geodesic. Then the following statements are equivalent: 
\begin{enumerate}
\item $\Sigma^g_{\gamma(s)}X$ is non-contractible for some $s\in(0,1)$
\item $\Sigma^g_{\gamma(s)}X$ is non-contractible for all $s\in(0,1)$
\item $\Sigma^g_{\gamma(s)}X\backslash\{\pm\dot{\gamma}(s)\}$ is non-contractible
for some  $s\in(0,1)$
\item $\Sigma^g_{\gamma(s)}X\backslash\{\pm\dot{\gamma}(s)\}$ is non-contractible
for all $s\in(0,1)$

\end{enumerate}

\end{proposition}
\begin{proof}

Equivalencies  $(1)\Leftrightarrow (2)$, $(3)\Leftrightarrow (4)$  
immediately follow  from Lemma~\ref{lem-kramer}.

{Let us prove $(1)\Leftrightarrow (3)$. By Lemma~\ref{lem-geod-noncotr-2} either both  $\Sigma^g_{\gamma(s)}X$ and $\Sigma^g_{\gamma(s)}X\backslash\{\pm\dot{\gamma}(s)\}$ are contractible or every point in $\Sigma^g_{\gamma(s)}X$ has an opposite. In the latter case  by Corollary \ref{cor-cd-cut-susp} $\Sigma^g_{\gamma(s)}X\cong \SS^k$ for some $k<N$ and hence  $\Sigma^g_{\gamma(s)}X\backslash\{\pm\dot{\gamma}(s)\}$ is homotopy equivalent to $\SS^{k-1}$. In particular both $\Sigma^g_{\gamma(s)}X$ and $\Sigma^g_{\gamma(s)}X\backslash\{\pm\dot{\gamma}(s)\}$ are non-contractible. This establishes  $(1)\Leftrightarrow (3)$.}



\end{proof}

\begin{proposition}\label{prop:reg-pooints}
Let $X\in \CC$. 
\begin{enumerate}[(i)]
\item \label{g-reg-implies-man} Let $p\in \greg_m$ for some $m\in \N$.

Then $m\le N$ and $p$ is inner, $p\in\Xreg_m$ and an open neighbourhood  of $p$ is homeomorphic to $\R^m$. 

\item \label{greg-iff-reg} $\Xreg_m=\greg_m$ for any $m$.
\item\label{reg-iff-inner} $q\in \Xreg$ if and only if $q$ is inner.
\item\label{non-contr-sp-dir} $q\in \Xreg$ if and only if $\Sigma_q^gX$ is non-contractible.

\item \label{man-implies-g-reg} If an open neighborhood $W$ of $p$ is homeomorphic to $\R^m$, then $W\subset \Xreg_m$.

\end{enumerate}

\end{proposition}
\begin{proof}
Let us first prove part \eqref{g-reg-implies-man}.
Suppose $T_p^gX\cong \R^m$.  Since $T_p^gX\in \CC$ we must have that $m\le N$. By \cite[Theorem A]{Kramer11} there is a small $R>0$ such that $B_R(p)\backslash \{p\}$ is homotopy equivalent to $\SS^{m-1}$.
Since $\SS^{m-1}$ is not contractible, by~\cite[Theorem 1.5]{Lyt-Schr07} there is $0<\eps<\pi_\kappa/2$ such that every geodesic starting at $p$ extends to a geodesic of length $\eps$. 
The natural "logarithm" map $\Phi\co \bar B_\eps(p)\to \bar B_\eps(0)\subset T_p^gX$ is Lipschitz since $X$ is $\CAT(\kappa)$. By the above mentioned result of Lytchak and Schroeder~\cite[Theorem 1.5] {Lyt-Schr07} $\Phi$ is \emph{onto}.  

We also claim that $\Phi$ is one-to-one. If $\Phi$ is not one-to-one then 
there exist two distinct unit speed geodesics $\gamma_1,\gamma_2$ of the same length $\eps'\le \eps$ such that $p=\gamma_1(0)=\gamma_2(0)$, $\gamma_1'(0)=\gamma_2'(0)$ but $\gamma_1(\eps')\ne\gamma_2(\eps')$.

Let $v=\gamma_1'(0)=\gamma_2'(0)$. Since $T_p^gX\cong \R^m$ the space of directions  $T_p^gX$ contains the opposite vector $-v$. Then there is a geodesic $\gamma_3$ of length $\eps$ starting at $p$ in the direction $-v$.
Since  $X$ is $\CAT(\kappa)$ {and $2\eps<\pi_k$, the concatenation of $\gamma_3$ with $\gamma_1$ is a geodesic} and the same is true for $\gamma_2$. This  contradicts the fact that $X$ is non-branching. 
Thus, $\Phi$ is a continuous bijection and since both $\bar B_\eps(p)$ and $\bar B_\eps(0)$ are compact and Hausdorff it's a homeomorphism.

The above argument also shows that $p$ is inner and  all geodesics starting at $p$ are uniformly extendible. Therefore by  Lemma~\ref{geom-vs-blowup-t-cones}  $T_pX$ is unique and  is equal to $T_p^gX$ and hence $p\in \Xreg_m$.
 This proves part \eqref{g-reg-implies-man}.

Part \eqref{greg-iff-reg} follows by part \eqref{g-reg-implies-man} and Proposition~\ref{reg=reg-g}.

Next let us prove part \eqref{reg-iff-inner}.

Suppose $p\in \Xreg_m$. Then $x\in \greg_m$ by part \eqref{greg-iff-reg} and hence $p$ is inner by part \eqref{g-reg-implies-man}.
Conversely, suppose $p$ is inner. Then  $T_p^gX\cong \R^m$ by Lemma~\ref{lem-inner-tangent}. Therefore $p\in \Xreg_m$ by part  \eqref{greg-iff-reg}.

Next, let us prove \eqref{non-contr-sp-dir}. If $\Sigma_p^gX$ is non-contractible then by the above mentioned result of Lytchak and Schroeder ~\cite[Theorem 1.5]{Lyt-Schr07} $p$ is inner. Hence it's regular by part \eqref{reg-iff-inner}. Conversely, if $p\in\reg_m$ then $p\in \greg_m$ by part \eqref{greg-iff-reg}  and hence $\Sigma_p^gX\cong \SS^{m-1}$ which is not contractible.

Lastly, let us  prove part \eqref{man-implies-g-reg}. Suppose an open neighborhood $W$ of $p$ is homeomorphic to $\R^m$.  \\By ~\cite[Lemma 3.1]{Kap-Ket-18} (or by the same argument as above using \cite[Theorem A]{Kramer11}  and ~\cite[Theorem 1.5]{Lyt-Schr07} ) any $q\in W$ is inner. Therefore it's regular and geodesically regular. Hence by part  \eqref{g-reg-implies-man} an open neighbourhood of $q$ is homeomorphic to $\R^{l(q)}$ for some $l(q)\le N$. Since $W$  is homeomorphic to $\R^m$ this can only happen if $l(q)=m$.

\end{proof}

 In ~\cite{kleiner} Kleiner studied the following notion of dimension of  $\CAT$ spaces.
 
 Let $X$ be $\CAT(\kappa)$. Pick any $p\in X$. Let $\Sigma_1=\Sigma^g_pX$. Recall that  $\Sigma_1$ is $\CAT(1)$. For any $v_1\in\Sigma_1$ let $\Sigma_2$ be equal to $\Sigma^g_{v_1}\Sigma_1$. We can iterate this construction. If $\Sigma_k$ is already constructed and is non-empty pick $v_k\in \Sigma_k$ and set
 $\Sigma_{k+1}=\Sigma_{v_k}^g\Sigma_k$.
 Following Kleiner we adopt the following definition
 \begin{definition}
 Let $(X,d)$ be a $\CAT(\kappa)$ space. We define the \emph{splitting dimension} $\dims  X$ of $X$ to be the the $\sup k$ such that there exist a chain of the above form with all $\Sigma_k\ne\varnothing$.\footnote{Kleiner calls this the geometric dimension in ~\cite{kleiner}. We use a different term to avoid a clash of terminology with geometric dimension of $\RCD$ spaces.}
 \end{definition}
 Note that for general  $CAT(\kappa)$ spaces it's possible to have $\dims X=\infty$ even if $X$ is compact. Also, it's obvious that if $X$ is connected and not a point then $\dims X\ge 1$. Kleiner showed that $\dims\le\dimh$. Therefore all elements of $\CC$ have finite splitting dimension. For $X\in \CC$ we define the geometric dimension of $X$ $\dimg X$ to be the largest $k$ such that $\reg_k\ne\varnothing$. If all $\reg_k$ are empty we set $\dimg X=-1$. As we will see later this case can not occur but this is not obvious at the moment.

 Next we study the relations between various notions of dimension in case $X\in \CC$.
 
 We prove the following 
 
 \begin{theorem}\label{c-dim}
Let $X\in \CC$.
Then $\dimt X=\dimg X=\dims X\le \dimh X\le N$.

Moreover $\reg_{k}\ne\varnothing$ for $k=\dims X$.
 \end{theorem}
We conjecture that if $X$ satisfies  \eqref{eq:cd+cat} then in the above theorem the second to last inequality is always an equality i.e. $\dims X= \dimh X$ .  
 \begin{proof}
 It's well known that $\dimt\le  \dimh$ for general metric spaces and since $X\in \CC$ we have that $\dimt X\le \dimh X\le N$.
 
 By \cite[Theorem A]{kleiner} it holds that
 \[
 \dims X=\sup \{\dimt K| K\subset X \text{ is compact}\}
 \]
 
 Since $X$ is proper  $\dimt X=\sup \{\dimt K| K\subset X \text{ is compact}\}$ and therefore
\[
  \dims X=\dimt X\le \dimh X\le N
\]
 
 By \cite[Theorem B]{kleiner} for a $\CAT(\kappa)$ space $Y$ with finite splitting dimension it holds that
 \begin{equation}\label{eq:dim}
 \dims Y=\max\{k| \text{ there is } p\in Y \text{ with }  \tilde H_{k-1}(\Sigma_p^gY)\ne 0\}
 \end{equation}
 
 Since if $p\in\reg_n$ then $\Sigma_pX\cong \SS^{n-1}$  and $\tilde H_{n-1}(\SS^{n-1})\ne 0$ this implies that $\dims X\ge \dimg X$.
 
 Now let $k=\dims X$. Then by \eqref{eq:dim}  there is $p\in X$ such that  $\tilde H_{k-1}(\Sigma_p^g X)\ne 0$. Hence $\Sigma_p^g X$ is not contractible and therefore $p\in\greg_m$ for some $m$. Then $\tilde H_{k-1}(\SS^{m-1})\ne 0$ and hence $k=m$. This shows that $\dims X\le \dimg X$.
 
 Moreover this argument also shows that  $\reg_{k}\ne\varnothing$ which finishes the proof of the theorem.
 
 \end{proof}

 In view of Theorem ~\ref{c-dim} from now on for $X\in\CC$ we will not distinguish between $\dimt X$, $\dimg X$ and $\dims  X$ and will refer to any of these numbers as the dimension of $X$ which will denote by $\dim X$.

 \begin{lemma}\label{t-cone-dim}
For any $p\in X$  and any tangent cone $T_pX$ it holds that $\dim T_p^gX\le \dim T_pX\le \dim X$.
 \end{lemma}
 \begin{proof}
 The first inequality is obvious since $ T_p^gX\subset T_pX$. The second inequality is an immediate consequence of the definition of a tangent cone and   \cite[Theorem A]{kleiner} which shows that for a general $\CAT(\kappa)$ space $Y$ it holds that
 $\dims Y$ is equal to the supremum of all $k$ such that there exist $q\in Y, R_j\to 0, S_j\subset Y$  such that $d(S_j,q)\to 0$ and  $\frac{1}{R_j}S_j$ Gromov\textendash  Hausdorff converges to $\bar{B}_1(0)\subset \R^k$.
 \end{proof}
 \begin{remark}
We currently don't know any examples where any of the inequalities in the above lemma are strict.
 \end{remark}
\begin{remark}
For spaces satisfying  \eqref{eq:cd+cat} the inequality $\dimg T_pX\le \dimg X$ also follows from lower semicontinuity of  geometric dimension for $\RCD(K,N)$ spaces ~\cite{KitaPOTA}.
\end{remark}

\begin{theorem}\label{th:reg-points}
Let $X\in \CC$ and set $m=\dim X$. Then  
\begin{enumerate}[(i)]
\item \label{R=R_m} $\Xreg=\Xreg_m$.
\item \label{reg-conv-dense-open} $\Xreg$ is dense,  geodesically convex and open. 
\item \label{reg-str-convex}$\Xreg$ is strongly convex in the following sense: if $\gamma(t_0)\in \Xreg$ for some $t_0\in(0,1)$ then $\gamma(t)\in \Xreg$ for all $t\in(0,1)$.


\item \label{geod-non-extend-bry} If  $p\in \Xreg$ and $y\in X \backslash \Xreg$ then no geodesic $\gamma$ between
$p$ and $y$ can be locally  extended beyond $y$.

\item\label{uniform-ext-geod} For any compact set $C\subset \Xreg$ there is $\eps=\eps(C)>0$ such that every geodesic starting in $C$ can be extended to length at least $\eps$.
\end{enumerate}
\end{theorem}
\begin{proof}
By Theorem~\ref{c-dim} $\Xreg_m$ is non-empty and in Proposition \ref{prop:reg-pooints} we have shown that $\Xreg_k$ is open for any $k$. 

 Let $p$ and $q$ be two distinct points  with $p\in \Xreg_k$ for some $k$. Let $\gamma \co [0,1]\to X$ be a constant speed geodesic with $\gamma(0)=p$ and $\gamma(1)=q$. By Proposition~\ref{prop:reg-pooints} for $t$ close to $0$ it holds that $\Sigma_{\gamma(t)}^gX\cong \SS^{k-1}$ and in particular it is non-contractible. Therefore by Proposition~\ref{inner-iff-non-contr-iff-non-contr-minus-geod}  $\Sigma_{\gamma(t)}^gX$ is non-contractible for all $t\in [0,1)$. Hence $\gamma(t)\in\reg_{m(t)}$ for each $t\in [0,1)$ by Proposition \ref{prop:reg-pooints}.
 Since each $\reg_{m(t)}$ is open and $[0,1)$ is connected this can only happen if $\gamma(t)\in \reg_k$ for all $t\in [0,1)$.  This argument also shows that if $\gamma$ can be locally extended past $q$ then $q\in \reg_k$ as well.
This implies that $\reg$ is dense in $X$ and that a geodesic from a regular point to a point $q\in X\backslash \reg$ can not be locally extended past $q$. This proves parts \eqref{reg-conv-dense-open}, \eqref{reg-str-convex} and \eqref{geod-non-extend-bry}.

 Furthermore the same proof shows that if $\gamma(0)\in \reg_k$ and $\gamma(1)\in \reg_l$ then $k=l$. Thus there is only one $k$   such that $\reg_k\ne\varnothing$. Since we've shown that $\Xreg_m\ne\varnothing$ this $k$ must be equal to $m$. This proves \eqref{R=R_m}.

Finally, part \eqref{uniform-ext-geod} immediately follows from above and compactness of $C$.

\end{proof}
\begin{remark}
Note that once the equivalence of Corollary \ref{inner-iff-non-contr-iff-non-contr-minus-geod} is proven it is possible 
to show that $\Xreg$ is a geodesically convex open smooth manifold with a $C^0$ Riemannian metric using Berestovski\u{\i}'s argument in \cite{berestovskii02}.
As we will see later,  using recent work of Lytchak and Nagano for spaces satisfying  \eqref{eq:cd+cat} this can be improved to show that this Riemannian metric  is $ \BV_0$.
\end{remark}

\section{Boundary}\label{sec:boundary}
In this section we introduce the notion of the boundary of spaces in $\CC$ (in particular for $\RCD$+$\CAT$ spaces) and study its properties.

Let $X\in \CC$ and $p\in X$. Since $T_p^gC\in \CC$,  it is non-branching and therefore $\Sigma_p^gX$ has at most two components and the only way it can have two components is if both are points.

Further, if $\dim T_p^gX=1$ then $T_p^gX$ must be isometric to either $\R$ or $[0,\infty)$. We'll say that $T_p^gX$ has boundary  (equal to $\{0\}$) in the latter case but has no boundary in the former case. 

We'll say that $T_p^gX$ of $\dim> 1$ has boundary if there is $v\in \Sigma_p^gX$ such that $T_{v}^g\Sigma_p^gX$ has boundary. This definition makes sense since  $\dim T_v^g\Sigma^g_pX<\dim T_p^gX\le \dim X$.

 For $X\in \CC$ of $\dim\ge 1$ we define the boundary $\bCAT X$ as the set of all points $p\in X$ such that $T_p^gX$ has boundary. 
 Lastly if $\dim X=0$ (this can only occur if $X=\{pt\}$) we set $\bCAT X=\varnothing$.

Note that this definition is analogous to the definition of the boundary of Alexandrov spaces and to the definition of the boundary of non-collapsed $\RCD(K,N)$ spaces introduced in~\cite{KapMon}. All three definitions agree if 
 $X$ satisfies  \eqref{eq:cd+cat} and $\dim X=N$ (such space is automatically Alexandrov by Corollary~\ref{cor:noncol} and \cite{Kap-Ket-18}).
 
 Obviously, if $p\in\greg$ then $p\notin \bCAT X$. We show that the converse is also true.
\begin{proposition}\label{int-impl-reg}

Let $X\in \CC$. 
Suppose $p\notin \bCAT X$. Then $p\in \greg$.
\end{proposition}
\begin{proof}
We only need to consider the case when $X$ is connected and is not a point. Then $\dims X\ge 1$.

We will prove by induction on $\dims X$ that if  $T_p^gX$ has no boundary then it's isometric to $\R^l$ for some $l\le \dims X$.

The base of induction $\dims T_p^gX=1$ was already discussed above. 

Induction step. Suppose $l>1$  and we've already  proved this for spaces of $\dim <l$. Suppose $p\in X$ and  $\dim T_p^gX=l$. 
Then for any $v\in \Sigma_p^gX$ we have that  $\dim T_v^g\Sigma^g_pX<l$. Further by the definition of boundary $\bCAT\Sigma_p^gX=\varnothing$.
Therefore by the induction assumption  $\Sigma^g_v\Sigma_p^gX$ is isometric to a round sphere of some dimension $d(v)-1\le l$ where $d(v)$ can a priori depend on $v$. 

Now by Proposition~\ref{prop:reg-pooints}
we get that a small neighborhood of $v$ in $\Sigma_p^gX$  is homeomorphic to $\R^ {d(v)}$. Since $v\in \Sigma_k$ was arbitrary this means that   $\Sigma_p^gX$  is a  closed manifold of dimension $\ge 1$.  Therefore $\Sigma_p^gX$is non-contractible. Therefore $p\in \greg$ by Proposition~\ref{prop:reg-pooints}.

\end{proof}

Combining the above proposition with   Proposition~\ref{prop:reg-pooints} we immediately obtain 

\begin{theorem}
Let $X\in \CC$. 

Then a point $p\in X$ belongs to $\bCAT X$ iff $p\in X\backslash \greg$ iff $\Sigma_p^gX$ is contractible. In particular $\bCAT X=X\backslash \greg$ is closed.

\end{theorem}

Next we show that spaces in $\CC$ are topological manifolds with boundary.
\begin{theorem}\label{c-manifold-with-boundary}
Let $X\in \CC$. Then $X$ is homeomorphic to an $m$-dimensional manifold
with boundary with $m=\dim X$. Furthermore, the manifold boundary $\partial X$ is equal to the geometric boundary $\bCAT X$.
\end{theorem}
\begin{proof}
By  Theorem~\ref{th:reg-points}  we know that $\greg \subset X$ is open and it is a connected manifold of dimension $m=\dim X$. Thus we only need to understand the topology of $X$ near boundary points.

Let $p\in\bCAT X$. We will show that  it admits an open neighborhood $U$  homeomorphic to $\mathbb{R}_{+}^{m}$. Recall that $X$ is $\CAT(1)$. Let $0<R<\pi/10$. Since $\greg$ is dense there is  $y\in\Xreg_m\cap B_{\frac{R}{2}}(p)$. Since  $0<R<\pi/10$ all geodesics in $\bar{B}_{R}(y)$ are unique. Furthermore, by   Proposition~\ref{prop:reg-pooints} there is $0<r<\min\{R/10, d(y,\bCAT X)/10\}$ such that the metric sphere $S_r(y)$ is homeomorphic to $\SS^{m-1}$.

By Proposition~\ref{prop:reg-pooints}  we know that any geodesic ending in a regular point can be locally extended past  that point and by Theorem~\ref{th:reg-points}  a geodesic starting at a regular point can not be locally extended past a boundary point. Further recall that in a $\CAT(1)$ space local geodesics of length $<\pi$ are geodesics. For any $z\in S_r(y)$ let $\gamma_z \co [0, f(z)]\to  B_R(y)$ be a maximal unit speed geodesic starting at $y$ and passing through $z$. Since $X$ is non-branching such $\gamma_z$ is unique. By above $\gamma(f(z))\in \bCAT X$ or $\gamma(f(z))\in S_R(y)$. 

Let $\Phi(z)=\gamma_z({ f(z)})$. By above if $f(z)<R$ then $\Phi(z)\in\bCAT X$. Let {$z_0=[y,p]\cap S_r(y)$, $ p\in \bCAT X$}. By construction $f(z_0)<{{R/2<}}R$.

{\bf Claim.} $f$ is continuous near $z_0$.

It's enough to prove continuity at $z_0$ since it will imply that $f(z)<R$ and hence $\Phi(z)\in \bCAT X$ for all $z$ close to $z_0$.

Suppose there is $z_i\in S_r(y)$ converging to $z_0$ such that $f(z_i)\to l<f(z_0)<{ R}$. Then the geodesics $[y,\Phi(z_i)]$ subconverge to a geodesic $[y,q]$ of length $l$ starting at $y$ and passing through $z_0$. By uniqueness and non-branching of geodesics $[y,q]$ is a part of the geodesic $[y,\Phi(z_0)]$.  Since $\bCAT X$ is closed we must have that $q\in \bCAT X$. But $q$ is an interior point of $[y,\Phi(z_0)]$. This is impossible by { Theorem \ref{th:reg-points} (iv)}. This is a contradiction and hence $f$ is lower semicontinuous at $z_0$.
Now suppose there is $z_i\in S_r(y)$ converging to $z_0$ such that $f(z_i)\to l>f(z_0)$. Then again $[y,\Phi(z_i)]$ subconverge to a geodesic $[y,q]$ of length $l$ and since $l>f(z_0)$ we must have that $z_0$ is an interior point of  $[y,q]$. This again is impossible since $p=\Phi(z_0)\in \bCAT X$. Thus $f$ is upper semicontinuous and hence continuous at $p$.  This finishes the proof of the Claim.

The claim immediately implies that $\Phi$ is continuous near $z_0$. Since  geodesics of length less than $\pi$ in a $\CAT(1)$ space are unique $\Phi$ is one-to-one near $z_0$. 

Next, by the $CAT(\kappa)$-condition the map $\Phi^{-1}$ is continuous (in fact,  Lipschitz) on $B_\eps(p)\cap \bCAT X$ for all small $\eps$. Therefore $\Phi$ is a homeomorphism from a small  neighborhood $U$ of $z_0$ in $S_r(y)$ to a small  neighborhood of $p$ in $\bCAT X$. Furthermore,
the map $\Psi:U\times \red{(\frac 1 2,1]}\to X$ defined by 
\[
\Psi(z,t)=\gamma_{z}(tf(z))
\]
is a homeomorphism onto a neighborhood of $p$  in $X$. This proves that $X$ is an $m$-manifold with boundary.

Let us verify that $\bCAT X=\partial X$. The inclusion $\partial X\subset \bCAT X$ follows from Proposition~\ref{prop:reg-pooints}\eqref{g-reg-implies-man}    since regular points have open neighborhoods homeomorphic to $\R^m$.
The inclusion $\bCAT X\subset \partial X$ follows  by Proposition~\ref{prop:reg-pooints}\eqref{man-implies-g-reg}   which says that if $p$ is in the manifold interior of $X$ then it is regular and hence does not lie in $\bCAT X$.

\end{proof}
\begin{corollary}\label{rcd+cat-imply-mnfld-bry}
Let $X$ satisfy  \eqref{eq:cd+cat} and set $n=\dim X$. Then $X$ is a  $n$-dimensional manifold with boundary and $\partial X=\bCAT X$.
\end{corollary}

\begin{theorem}[Sphere Theorem] \label{sphere-thm-c}
Suppose $\Sigma$ and $X=C(\Sigma)$ lie in $\CC$. Let  $m+1=\dim X$. Then the following dichotomy holds.

 If $\bCAT \Sigma=\varnothing$ then $\Sigma\cong \SS^m$.

If $\bCAT \Sigma\ne \varnothing$ then $\Sigma$ is homeomorphic to the closed disk $\bar D^m$.

\end{theorem}
\begin{proof}
Suppose $\bCAT \Sigma=\varnothing$. Then $X=C\Sigma$ has no boundary either and therefore every point in $X$ including the vertex $o$ is regular and $T_oX\cong \R^{m+1}$. But $T_oX\cong C(\Sigma)$. Therefore $\Sigma\cong \SS^m$.

Now suppose $\bCAT \Sigma\ne\varnothing$.  If $\diam \Sigma=\pi$ then for $p,q\in \Sigma$ with $d(p,q)=\pi$ we have that $\Sigma$ is isometric to the spherical suspension over $Y=\Sigma_p^g\Sigma\in \CC$ and $\bCAT Y\ne \varnothing$.

 Thus we can reduce the problem to the case when $\diam\Sigma<\pi$. Then geodesics between any two points in $\Sigma$ are unique and depend continuously on the endpoints. Let $p\in\greg_m$. There exists $R<\pi$ and $\eps>0$ such that $X=B_{R-\eps}(p)$.
Pick a small $r>0$ such that $\bar B_r(p)\subset \greg_m$ and $S_r(p)$ is homeomorphic to $\SS^{m-1}$. Let $\Phi\co S_r(p)\to \partial X$ and $f\co S_r(p)\to \R$ be the same maps as in the proof of Theorem \ref{c-manifold-with-boundary}. Then as before $f$ is continuous and $\Phi$ is a homeomorphism. Further, radially extending $\Phi $ to the closed unit ball around the vertex in $C(S_r(p))$ by the formula
\[
\Psi(t,z)=\gamma_z(tf(z)/r)
\]
we get a homeomorphism from $\bar D^m$ to $\Sigma$.
\end{proof}
\begin{corollary}\label{rcd+cat-sphere-thm}
Let $(X,d,m)$ be $\RCD(N-1,N)$ and $\CAT(1)$ where $N>1$. If $\bCAT X\ne \varnothing$  and $X$ is homeomorphic to a closed disk of dimension $\le N$.
On the other hand, if $\bCAT X = \varnothing$  then $N$ is an integer and $X$ is metric measure isomorphic to $(\SS^{N}, const\cdot \mathcal H_N)$.
\end{corollary}

\begin{proof}
Due to the Sphere Theorem we only need to prove the second part. If $\partial X=\varnothing$ then by the Sphere Theorem $X$ is isometric to $\SS^l$ with $l\le N$.
Since the metric measure cone over $X$ is $\RCD (0,N+1)$~\cite{Ket-cones}  and  is isometric to $\R^{l+1}$  by the splitting theorem it must be metric measure isomorphic to $(\R^{l+1}, const\cdot \mathcal H_{l+1})$. Therefore $m=const\cdot \mathcal H_{l}$.
This obviously implies that $l=N$.
\end{proof}

The Sphere Theorem immediately implies
\begin{corollary}
Let $X$ satisfy  \eqref{eq:cd+cat} and $p\in \bCAT X=\partial X$.
Then $T_p^gX$ is homeomorphic to $\R^l_+$ and $\Sigma_p^g$ is homeomorphic to $\bar D^{l-1}$ where $l\le \dim X$.
\end{corollary}
At the moment we don't know if in the above corollary $l$ must be equal to $\dim X$. 
\begin{question}
Is it true that for any $p\in \partial X$ it holds that $T_p^gX$ is homeomorphic to $\R^n_+$ where $n=\dim X$? Weaker, is $\dim T^g_pX$ locally constant on $\bCAT X$?
\end{question}

We conclude this section by studying the boundary as seen from a regular point. 
Let $\mathcal{C}_{0}\subset\mathcal{C}$ be the class of $\CAT(0)$ spaces
which split off lines. An easy observation is the following:
\begin{lemma}
If $(X,d)\in\mathcal{C}_{0}$ is non-compact then either $X$ is isometric
to a product of the real line and some compact $(X',d')\in\mathcal{C}_{0}$
with $\partial X'\ne\varnothing$ or $X$ has exactly one geodesic
end, i.e. for any $x_{n},y_{n}\in X\backslash\bar{B}_{R}(x)$ with
$x_{n},y_{n}\to\infty$ it holds 
$[x_{n},y_{n}]\cap \bar{B}_{R}(x)=\varnothing$ eventually.
In particular, any space $(X,d)\in\mathcal{C}_{0}$ with $\partial X=\varnothing$
is isometric to $\mathbb{R}^{m}$ where $m=\dim X$.
\end{lemma}
\begin{remark}
As a simple corollary we see that an end is also a geodesic end and
vice versa. 
\end{remark}
Let $(X,d)\in\mathcal{C}_{0}$ and pick a regular point $p\in\mathcal{R}$.
Define a function 
\[
f:\Sigma_{p}X\to(0,\infty]
\]
where  $f_{p}(v)$ is the length of the maximal unit-speed geodesic $\gamma$
issuing from $p$ with $\dot{\gamma}(0)=v$. It is not difficult to
verify that $f$ is bounded from below and continuous. Set
\begin{align*}
\Sigma_{p}^{\mathrm{fin}}X & =\{f_{p}<\infty\}\\
\Sigma_{p}^{\mathrm{inf}}X & =\{f_{p}=\infty\}
\end{align*}
and note $\Sigma_{p}^{\mathrm{fin}}X$ is open and $\Sigma_{p}^{\mathrm{inf}}X$
is closed. 
\begin{theorem}
Either $\Sigma_{p}^{\mathrm{inf}}X$ is connected or $X$ splits off
a line.
\end{theorem}
\begin{proof}
If $\Sigma_{p}^{\mathrm{inf}}X$ is not connected then there are two
disjoint open sets $U,V\subset\Sigma_{p}X$ with $\Sigma_{p}^{\inf}X\subset U\cup V$.
Since $\Sigma_{p}X$ is compact the set $A=\Sigma_{p}X\backslash(U\cup V)\subset\Sigma_{p}^{\mathrm{fin}}X$
is compact. 

Let $K\subset X$ be the subset of points on unit-speed geodesics
$\gamma$ with $\dot{\gamma}(0)\in A$. By compactness of $A$ and
continuity of $f$ we see that $f$ is uniformly bounded on $A$ so
that $K$ is closed and bounded. In particular, it is compact. 

Now let $\gamma$ and $\eta$ be unit speed geodesics with $\dot{\gamma}(0)\in\Sigma_{p}^{\inf}X\cap U$
and $\dot{\eta}(0)\in\Sigma_{p}^{\inf}X\cap V$. Let $\xi:[0,1]\to X$
be a geodesic connecting the endpoints of $\gamma$
and $\eta$. Then either $p\in\xi([0,1])$ or there is a continuous
curve $\rho:[0,1]\to\Sigma_{p}X$ such that $\rho(t)=\dot{\zeta}^{t}(0)$
where $\zeta^{t}$ is a unit speed geodesic connecting $p$ and $\xi(t)$.
Since $\rho(0)\in U$ and $\rho(1)\in V$ there must be a $t\in(0,1)$
with $\rho(t)\in A$ implying $\xi(t)\in K$. By the previous lemma
$X$ must split off a line. 
\end{proof}


{
\section{$\DC$ coordinates on $\RCD$+$\CAT$ spaces.} \label{sec:dc}
{
\subsection{$\DC$ coordinates on $\CAT$ spaces}\label{ln}
{In \cite{Lytchak-Nagano18} Lytchak and Nagano developed a structure theory of finite dimensional $\CAT$ spaces with locally extendible geodesics. In particular, they constructed $\DC$ coordinates on such spaces. This mirrors results of Perelman from  \cite{Per-DC} where a similar theory was developed for Alexandrov spaces with curvature bounded below.

In this paper we will only need the following special case of the Lytchak-Nagano theory.

Let $(X,d)$ be a $\CAT(\uk)$ space. Let $\hat{U}\subset X$ be open and suppose $\hat{U}$ is a topological $n$-manifold.
It is well known (see e.g. ~\cite[Lemma 3.1]{Kap-Ket-18} or the proof of Proposition~\ref{prop:reg-pooints}\eqref{g-reg-implies-man})
  that this implies that  geodesics in $\hat{U}$ are locally extendible. Suppose further that for any $ \hat U\subset \greg_n$, i.e. $T_q^gX\cong \R^n$ for any $q\in \hat U$.}

Then by ~\cite{Lytchak-Nagano18}  for any $p\in \hat U$  there exist $\DC$ coordinates near $p$ with respect to which the distance on $\hat U$ is locally induced by a  $\BV_0$-Riemannian metric $g$.

More precisely, let $a_1,\ldots, a_n, b_1,\ldots, b_n$ be points near $p$ such that $d(p,a_i)=d(p,b_i)=r$, $p$ is the midpoint of $[a_i,b_i]$ and $\angle a_ipa_j=\pi/2$ for all $i\ne j$ and all comparison angles $\tilde \angle a_ipa_j, \tilde\angle a_ipb_j,  \tilde\angle b_ipb_j$ are sufficiently close to $\pi/2$ for all $i\ne j$.

Let $x\co \hat U\to \R^n$ be given by $x=(x_1,\dots,x_n)=(d(\cdot,a_1),\ldots, d(\cdot, a_n))$.

Then by  ~\cite[Corollary 11.12]{Lytchak-Nagano18}  for any sufficiently small $0<\eps<\pi_k/4$ the restriction $x|_{B_{2\eps(}p)}$ is bi-Lipschitz onto an open subset  of $\R^n$. Let $U=B_\eps(p)$ and $V=x(U)$.
By  ~\cite[Proposition 14.4]{Lytchak-Nagano18} $x\co U\to V$ is a $\DC$ equivalence in the sense that $h\co U\to \R$ is $\DC$ iff $h\circ x^{-1}$ is DC on $V$.

Further, by  ~\cite[Theorem 1.2]{Lytchak-Nagano18} the distance  on $U$  is induced by a $\BV_0$ Riemannian metric $g$ which in $x$ coordinates is given by a $2$-tensor $g^{ij}(p)=\cos \alpha_{ij}$ where 
$\alpha_{ij}$ is the angle at $p$ between geodesics connecting $p$ and $a_i$ and $a_j$ respectively. By the first variation formula $g^{ij}$ is the derivative of $d(a_i, \gamma(t))$ at $0$ where $\gamma$ is the geodesic with $\gamma(0)=p$ and $\gamma(1)=a_j$. Since $d(a_i,\cdot)$, $i=1,\dots n$, are Lipschitz, $g^{ij}$ is in $L^{\infty}$. We denote $\langle v,w\rangle_g(p)= g^{ij}(p)v_iw_j$ the inner product of 
$v,w \in \mathbb R^n$ at $p$. The Riemannian metric $g^{ij}$ induces a distance function $d_g$ on $V$ such that $x$ is a metric space isomorphism for $\epsilon>0$ sufficiently small.

If $u$ is a Lipschitz function on $U$, $u\circ x^{-1}$ is a Lipschitz function on $V$, and therefore differentiable $\mathcal{L}^n$-a.e. in $V$ by Rademacher's theorem. 
Hence,  we can define the gradient of $u$ at points of differentiability of $u$ in the usual way as the metric dual of its differential. Then the usual Riemannian formulas hold and 
$\nabla u=g^{ij}\frac{\partial u}{\partial x_i}\frac{\partial}{\partial x_j}$ and $|\nabla u|^2_g=g^{ij}\frac{\partial u}{\partial x_i}\frac{\partial u}{\partial x_j}$ {a.e.\ .
}
}
\medskip

Let $\tilde {\mathcal A}$ be the algebra of functions of  
the form $\phi(f_1,\ldots, f_m)$ where $f_i=d(\cdot, q_i)$ for some $q_1,\ldots, q_m$ with $|q_ip|>\eps$ and $\phi$ is smooth.

Together with the first variation formula for distance functions continuity of $g$ implies that for any $u,h\in \tilde {\mathcal A}$ it holds that $\langle \nabla u, \nabla h\rangle_g$ is continuous on $V$. 

Furthermore, since $\frac{\partial}{\partial x_i}=\sum_jg_{ij}\nabla x_j$ where $g_{ij}$ is the pointwise inverse of $g^{ij}$, continuity of $g$ also implies that  any $u\in \tilde {\mathcal A}$ is $C^1$ on $V$. Hence, any such $u$ is $\DC_0$ on $V$.  {Therefore $\hat U$ can be given  a natural structure of $\DC_0$ (and in particular $C^1$) $n$-dimensional manifold with an atlas given by $\DC$ coordinate charts and $g$ is $\BV_0$ with respect to this  $\DC_0$ structure.}

By the same argument as in ~\cite[Section 4]{Per-DC} (cf. \cite{palvs}, \cite{ambrosiobertrand})  it follows that any $u\in \tilde {\mathcal A}$ lies in  $D({\bf \Delta},U,\mathcal H_n)$  and the $\mathcal{H}^n$-absolutely continuous part of ${\bf \Delta}_0 u$ can be computed using standard Riemannian geometry formulas; that is 
\begin{align}\label{equ:laplace}
[{\bf\Delta}_0u]_{ac}=\frac{1}{\sqrt{|g|}}\frac{\partial^{ap}}{\partial x_j}\biggl( g^{jk}\sqrt{|g|}\frac{\partial u}{\partial x_k}\biggr)
\end{align}
where $|g|$ denotes the pointwise determinant of $g^{ij}$.  Here ${\bf \Delta} _0$ denotes the measure valued Laplacian on $(U, d, \mathcal H_n)$.
Note that $g$, $\sqrt{|g|}$ and $\frac{\partial u}{\partial x_i}$ are $\BV_0$-functions, and the
derivatives on the right are understood as approximate derivatives.

{The definition of ${\bf \Delta}_0$ is analogous to the definition of the measure valued Laplacian for $\RCD$ spaces. In this case the inner product $\langle \cdot, \cdot\rangle$ that was introduced before Definition \ref{def:rcd} is replaced by $\langle \cdot,\cdot\rangle_g$. However, we will also see in Lemma \ref{lem:lipschitz} below that assuming \eqref{eq:cd+cat} $\langle \cdot,\cdot\rangle$ and $\langle \cdot,\cdot\rangle_g$ coincide for Lipschitz functions.}

\subsection{$\DC$ structure on $\RCD$+$\CAT$ spaces}

Now suppose that  $X$ satisfies  \eqref{eq:cd+cat}. Let $n=\dim X$
and let $\Xreg$ be the set of regular points $p$ in $X$. Recall that by Theorem~\ref{th:reg-points}  $\Xreg=\Xreg_n$ and it is open. Further, by Proposition~\ref{prop:reg-pooints} for any $p\in \reg$ there is an open  neighbourhood $\hat U$ of $p$ in $\Xreg$ homeomorphic to $\R^n$. 

Thus all of the theory from Subsection \ref{ln} applies with $\hat U=\reg$ and we obtain:

\begin{theorem}\label{thm-DC-metric-structure}
 Let $X$ satisfy \eqref{eq:cd+cat} and let $n=\dim X$. Then $\reg$ admits the structure of an $n$-dimensional $\DC_0$ manifold (and in particular a $C^1$ manifold). Furthermore $\reg$ admits a $ BV_0$ Riemannian metric $g$ which induces the original distance $d$ on $\reg$. \end{theorem}

Recall that for a Lipschitz function $u$ on $V$ we have two a-priori different notions of the norm of the gradient defined $m$-a.e.: the "Riemannian"  norm of the gradient $|\nabla u|^2_g=g^{ij}\frac{\partial u}{\partial x^i}\frac{\partial u}{\partial x_j}$ and 
 the minimal weak upper gradient $|\nabla u|$  when $u$ is viewed as a  Sobolev functions in $ W^{1,2}(\m)$. We observe that these two notions are equivalent.


\begin{lemma}\label{lem:lipschitz}
Let $u, h: U\rightarrow \mathbb{R}$ be  Lipschitz functions. Then $|\nabla u|=|\nabla u|_g$, $|\nabla h|=|\nabla h|_g$ $\m$-a.e. and $\langle \nabla u,  \nabla h\rangle =\langle \nabla u,  \nabla h\rangle_g$ $\m$-a.e..

In particular, $g^{ij}=\langle \nabla x_i,\nabla x_j\rangle_g=\langle \nabla x_i,\nabla x_j\rangle$ {$\m$-a.e..}
\end{lemma}
\begin{proof}
First note that since both  $\langle \nabla u, \nabla h\rangle$ and $\langle \nabla u,  \nabla h\rangle_g$ satisfy the parallelogram rule, it's enough to prove that $|\nabla u|=|\nabla u|_g$ a.e..

Recall that $g^{ij}$ is continuous on $U$.
Fix a point $p$ where $u$ is differentiable. Then
\begin{align*}
\Lip u(p)= \limsup_{q\rightarrow p}\frac{|u(p)-u(q)|}{d(p,q)}&= \limsup_{q\rightarrow p}\frac{|u(p)-u(q)|}{|p-q|_{g(p)}}\\
&=\sup_{|v|_{g(p)}=1} D_v u=\sup_{|v|_{g(p)}=1} \langle v, \nabla u\rangle_{g(p)}=|\nabla u|_{g(p)}.
\end{align*}
In the second equality we used that $d$ is induced by $g^{ij}$, and that $g^{ij}$ is continuous.
Since $(U,d,\m)$ admits a local 1-1 Poincar\'e inequality and is doubling, the claim follows from \cite{cheegerlipschitz} where it is proved that for such spaces $\Lip u=|\nabla u|$ a.e..
\end{proof}

In view of the above Lemma from now on we will not distinguish between  $|\nabla u|$ and $ |\nabla u|_g$ and between $\langle \nabla u, \nabla h\rangle$ and $\langle \nabla u,  \nabla h\rangle_g$.

\section{Density functions}\label{sec: dens}

\subsection{Non-collapsed case}
Let $(X,d, f\mathcal H_n)$ be $\RCD(\ke,n)$ and $\CAT(\uk)$ where $0\le f\in L^1_{loc}(\mathcal{H}^n)$.

\begin{remark}
If $(X,d,\m)$ is a weakly non-collapsed $\RCD$ space in the sense of \cite{GP-noncol} or a space satisfying  the generalized Bishop inequality in the sense of \cite{kbg} and if $(X,d)$ is $\CAT(\uk)$, the assumptions are satisfied by \cite[Theorem 1.10]{GP-noncol}.
\end{remark}
Following Gigli and  De Philippis~\cite{GP-noncol} for any $x\in X$ we consider the monotone quantity $\frac{m(B_r(x))}{v_{k,n}(r)}$ which is non increasing in $r$ by the Bishop-Gromov volume comparison. 
Let $\theta_{n,r}(x)=\frac{m(B_r(x))}{\omega_nr^n}$. Consider the density function $[0,\infty]\ni \theta_{n}(x)=\lim_{r\to 0}\theta_{n,r}(x)=\lim_{r\to 0}\frac{m(B_r(x))}{\omega_nr^n}$.

 Since $n$ is fixed throughout the proof we will drop the  subscripts $n$  and from now on use the notations $\theta(x)$ and $\theta_{r}(x)$ for $\theta_{n}(x)$ and $\theta_{n,r}(x)$ respectively.

By Theorem \ref{th:reg-points} and \cite[Theorem 1.10]{GP-noncol} we have that $m$-almost all points $p\in X$ are regular and $\theta(x)=f(x)$ a.e. with respect to $m$.

Therefore we can and will assume from now on that $f=\theta$ everywhere.
 
\begin{remark}
Monotonicity of  $r\mapsto \frac{m(B_r(x))}{v_{k,n}(r)}$ immediately implies that $f(x)=\theta(x)>0$ for all $x$.
 \end{remark}
 
By rescaling the metric and lowering $K$ we can assume that $\uk=1$ and $K=-(n-1)$.
\begin{lem}\label{lem-theta-est}
There  are constants $C>0$ and  $0<R_0<\pi/100$ such that the following holds.

Let $(X,d, f\mathcal H_n)$ be $\RCD(-(n-1),n)$ and $\CAT(1)$
and let 
$\gamma:[0,1]\to X$ is a geodesic in $B_{2R}(x_{0})$ with $R<R_0$.

Then for all $r\in(0,R)$ and $t\in(0,1)$ it holds 
\[t\cdot  \m(B_{r}(\gamma(1)))^{\frac{1}{n}}\le
(1+CR^{2})\cdot\m(B_{t r}(\gamma(t)))^{\frac{1}{n}}
\]

Moreover, $\theta_{n}(\gamma(1))^{\frac{1}{n}}\le(1+4CR^{2})\theta_{n}(\gamma(t))^{\frac{1}{n}}$.\end{lem}
\begin{proof}
 Let $A=B_{r}(\gamma(1))$
and 
\[
A_{t,\gamma(0)}=\{\xi(t)\,|\,\xi\text{ is a geodesic between \ensuremath{\gamma(0)} and a point in  \ensuremath{A}}\}.
\]
Since $R_0<\pi/100$ by the $\mathrm{CAT}(1)$ condition (cf. \cite[Lemma 5.5]{Kap-Ket-18})
we get that 
\[
A_{t,\gamma(0)}\subset B_{tr}(\gamma(t)).
\]

A Taylor expansion argument in $R$ at $R=0$ shows that there are $C_1>0$ and $R_1>0$ such that $\sigma_{K,n}^{(t)}(R)\ge(1-C_{1}R^{2})t>0$ for any $t\in[0,1]$ and $0<R\le R_1$.
Also note that since $K=-(n-1)<0$ the function $\theta\mapsto\sigma_{K,n}^{(t)}(\theta)$ is monotone decreasing.

Combining the above with the Brunn\textendash Minkowski inequality
with $A_{0}=\{x\}$ and $A_{1}=A$ yields
\begin{align*}
(1-C_{1}R^{2})t\m(A)^{\frac{1}{n}} & \le\sigma_{K,n}^{(t)}(R)\m(A)^{\frac{1}{n}}\\
 & \le\m(A_{t,\gamma(0)})^{\frac{1}{n}}\le\m(B_{tr}(\gamma(t)))^{\frac{1}{n}}.
\end{align*}
Let  $C=2C_1$, and $R_0=\min\{R_1,\frac{1}{\sqrt{2C_1}}\}$. Then for any $0<R<R_0$ we have
\[
0<\frac{1}{1-C_1R^2}<1+CR^2
\]
and therefore

\[
t\cdot \m(A)^{\frac{1}{n}}\le \frac{1}{1-C_1R^2}\m(B_{tr}(\gamma(t)))^{\frac{1}{n}}\le (1+CR^2)\m(B_{tr}(\gamma(t)))^{\frac{1}{n}}
\]
which yields the fist claim in the lemma.
The last claim is obtained by dividing the above  inequality
by $t\cdot r$ and taking the limit as $r\to 0$.
\end{proof}
{
\begin{corollary}\label{cor-loc-bounded}
Let $(X,d, f\mathcal H_n)$ be $\RCD(-(n-1),n)$ and $\CAT(1)$. Then $\theta_n$ is locally bounded.
\end{corollary}
\begin{proof}
Since $\Xreg$ has full measure and $\theta_n \in L^1_{loc}(\mathcal H^n)$, $\{x\in \Xreg: \theta(x)<\infty\}$ has full measure. By extendability of geodesics at regular points $\theta_n\leq (1+4CR^2)^n\theta(x_0)$ on $B_{2R}(x_0)$ for $x_0\in \Xreg$ with $\theta_n(x_0)<\infty$, $R\in (0,R_0)$ and $R_0>0$ as in the previous lemma. Then, the claim follows since $\Xreg$ is also dense.
\end{proof}
}

\begin{cor}\label{cor:theta-const}
The function $\theta_{n}$ is constant in the interior of every geodesics
$\xi\co [0,1]\to X$, i.e it holds $\theta_{n}(\xi(t))=\theta_{n}(\xi(s))$ for
$t,s\in(0,1)$.
\end{cor}
\begin{proof}
Without loss of generality we can assume that length of $\gamma$ is smaller than $R_0$ given by Lemma~\ref{lem-theta-est}.
Then by Corollary~\ref{cor-loc-bounded} there is $D>0$ such that $\theta_n(\gamma(t))\le D$ for all $t\in [0,1]$.

Choose points $x,y\in\xi((0,1))$. We may assume that either $x$
or $y$ assumes the role of $\gamma(t)$ in the lemma above and the
other the role of $\gamma(1)$. In the lemma above we may choose
$R=2d(x,y)$ and obtain
\[
\theta_{n}(x)^{\frac{1}{n}}\le(1+2Cd(x,y)^{2})\theta_{n}(y)^{\frac{1}{n}}
\]
so that 
\[
\frac{\theta_{n}(x)^{\frac{1}{n}}-\theta_{n}(y)^{\frac{1}{n}}}{d(x,y)}\le\frac{8Cd(x,y)^{2}}{d(x,y)}\theta(y)^{\frac{1}{n}}.
\]
Exchanging the roles of $x$ and $y$ 
\[
\frac{|\theta(y)^{\frac{1}{n}}-\theta(x)^{\frac{1}{n}}|}{d(x,y)}\le\frac{8Cd(x,y)^{2}}{d(x,y)}\max\{\theta(x)^{\frac{1}{n}},\theta(y)^{\frac{1}{n}}\}\le {8CDd(x,y)}
\]
for another constant $D>0$.
Thus the function $F:(0,1)\to\mathbb{R}$ defined by $F(t)=\theta^{\frac{1}{n}}(\xi(t))$ satisfies
\[
F'(t)=0 \quad \text{ for  all }t\in (0,1)
\]
and therefore $F\equiv\mathsf{const}$ on $(0,1)$.
\end{proof}
\begin{cor}\label{cor:noncol}
Let $(X,d,f\mathcal{H}^{n})$ be $\RCD(\ke,n)$ and $\CAT(\kappa)$.
Then $f\equiv\theta_{n}\equiv\mathsf{const}$ almost everywhere. 
\end{cor}
\begin{proof}
Just note that by Proposition~\ref{prop:reg-pooints} a geodesic connecting  two regular points $x$
and $y$ can be extended to a local geodesic so that $x$ and $y$ are in the interior
of that local geodesic. By Corollary~\ref{cor:theta-const} this implies $\theta(x)=\theta(y)$ and proves the
result.
\end{proof}
\begin{rem}
The result is also true for weakly non-collapsed $\MCP(\ke,n)$+$\CAT(\kappa)$ spaces.
We postpone the proof to a subsequent work as it relies on adjusted
versions of the splitting theorem, Propositions
\ref{prop:reg-pooints} and  \cite[Theorem 1.10]{GP-noncol}.
\end{rem}



%


\subsection{General case} \label{sec:noncol}

Let $(X,d,\m)$ be $\RCD(\ke,N)$ and $\CAT(\uk)$ where $K,\uk\in \R, N\ge 1$.  Let $n=\dim X$. Recall that $\Xreg=\reg_n$.
Denote
\begin{align*}
\Xreg_*:= \left\{x\in \Xreg: \exists \lim_{r\rightarrow 0^+} \frac{\m(B_r(x)}{r^n\omega_n}\in (0,\infty)\right\}.
\end{align*}
By results in \cite{gipa}, \cite{kellmondino}, \cite{DeP-Mar-Rin} $\m(\Xreg\backslash \Xreg_*)=0$ (and hence $\Xreg_*$ has full measure) and $\m|_{\Xreg_*}$ and $\mathcal H^n|_{\Xreg_*}$ are mutually absolutely continuous and 
\begin{align*}
\lim_{r\rightarrow 0^+}\frac{\m(B_r(x))}{\omega_nr^n}= \frac{d \m|_{\Xreg_*}}{d \mathcal H^n|_{\Xreg_*}}(x)=: \theta(x)
\end{align*}
for $\m$-a.e. $x\in \Xreg_*$. 
\begin{proposition}\label{prop-semiconcave}
$\theta$ admits an a.e. modification $\hat \theta\co X\to [0,\infty)$ such that 
\begin{align*}
\theta(\gamma(t))^{\frac{1}{N}}\geq \sigma_\eta^{(t)}(l)
\theta(\gamma(0))^{\frac{1}{N}}+  \sigma_\eta^{(t)}(l)\theta(\gamma(1))^{\frac{1}{N}}.
\end{align*}
for any geodesic $\gamma$ and some $\eta=\eta(\kappa, K, N, n)<0$. In particular, ${\hat \theta^{\frac 1 N}}$ is semi-concave.
\end{proposition}
\begin{proof}
We extend $\theta$ via 
\begin{align*}
 [0,\infty]\ni \liminf_{r\rightarrow 0} \frac{\m(B_r(x))}{\omega_n r^n}=\hat{\theta}(x)
\end{align*}
to a function that is defined everywhere on $X$. { Note that at this point we have not ruled out the possibility that $\hat \theta$ might take the value $\infty$.}

Let $\gamma$ be geodesic with length $l< L <\frac{\pi_\kappa}{100}$.
By the upper curvature bound and the Brunn\textendash Minkowski inequality (see the proof of Lemma 5.5 in \cite{Kap-Ket-18} ) one shows that 
\begin{align*}
(1+c_2l^2)\m(B_{r(1+c_1 l^2)}(\gamma(1/2)))\geq \frac{1}{2}\m(B_r(\gamma(0)))+  \frac{1}{2}\m(B_r(\gamma(1)))
\end{align*}
for some constants $c_1=c_1(\kappa)$ and $c_2=c_2(K,N,{ L} )$. We set $c(\kappa, K, N)=\max\{c_1,c_2\}$. By Taylor expanding $\sigma_\eta^{(1/2)}(l)$ w.r.t. $l$ one can see that
there exists some $\eta(\kappa,K,N,n):=\eta<0$ and $L>0$ such that for any $0<l<L$ it holds that
\begin{align*}
\frac{1}{2}\frac{1}{(1+c l^2)^{\frac{n}{N}+1}}\geq \sigma_\eta^{(1/2)}(l)
\end{align*}
and therefore
\begin{align*}
\left(\frac{\m(B_{r(1+c_1 l^2)}(\gamma(1/2)))}{\omega_n r^n (1+c l^2)^{n} }\right)^{\frac{1}{N}}\geq \sigma_\eta^{(1/2)}(l)
\left(\frac{\m(B_r(\gamma(0)))}{\omega_n r^n}\right)^{\frac{1}{N}}+  \sigma_\eta^{(1/2)}(l)\left(\frac{\m(B_r(\gamma(1)))}{\omega_n r^n}\right)^{\frac{1}{N}}
\end{align*}
Hence, for $r\rightarrow 0$ and $t=1/2$ we obtain
\begin{align}\label{eq:theta-semiconcave}
\hat \theta(\gamma(t))^{\frac{1}{N}}\geq \sigma_\eta^{(t)}(l)
\hat \theta(\gamma(0))^{\frac{1}{N}}+  \sigma_\eta^{(1-t)}(l)\hat \theta(\gamma(1))^{\frac{1}{N}}
\end{align}
for any geodesic $\gamma$ with length $l$ less than $L$. 

Now, we observe that for any point $x\in X$ there exists at least one geodesic $\gamma$ of length less than $L$ such that $\gamma(0)=x$ and $\hat \theta(\gamma(1/2))<\infty$.  By above this implies that
$\hat \theta(x)<\infty$ as well.

{It's easy to see that the same proof that works for $t=1/2$ shows that  \eqref{eq:theta-semiconcave} holds for all $t\in[1/4,3/4]$ (we are still assuming that length of $\gamma$ is smaller than $L$).
Let $J\subset [0,1]$ be the set of points $t$ for which \eqref{eq:theta-semiconcave} holds. We know that $J$ contains $[1/4,3/4],\{0\},\{1\}$. Algebraic properties of $\sigma_\eta^{(t)}(l)$ imply that $J$ is closed under taking midpoints. This immediately implies that $J=[0,1]$.

 }
\end{proof}

\begin{remark} One can also show that 
$-\log \hat \theta$ is $\frac{K-8nc}{2}$-convex.
Indeed, since $\RCD(\ke,N)$ implies $\RCD(\ke,\infty)$,
one obtains that
\begin{align*}\label{ineq-1}
&-\log \frac{\m(B_{r(1+cl^2)}(\gamma(t))}{\omega_n r^n(1+cl^2)^n} + \frac{K}{2}t(1-t) W_2(\mu^r_0,\mu^r_1)^2\nonumber\\
&\ \ \ \ \ \ \ \ \ \ \ \ \ \ \ \ \ \ \ \ \ \ \ \ \leq -(1-t)\log \frac{\mu(B_r(\gamma(0)))}{\omega_nr^n} - t \log \frac{\mu(B_r(\gamma(1)))}{\omega_nr^n} +n\log(1+cl^2)
\end{align*}
where
$\mu^r_i=\m(B_r(\gamma(0)))^{-1}\m|_{B_r(\gamma(0))}$, $i=0,1$. 
Moreover, $\log(1+cl^2)\leq cl^2$. 
Since $\mu_i^r\rightarrow \delta_{\gamma(i)}$, $i=0,1$ w.r.t. $W_2$, 
\begin{align*}
\lim_{r\rightarrow 0} W_2(\mu_0,\mu_1)=W_2(\delta_{\gamma(0)},\delta_{\gamma(1)})= d(\gamma(0),\gamma(1)).
\end{align*}
Hence, for $t=\frac{1}{2}$
\begin{align*}
-\log \hat \theta(\gamma\left(\frac{1}{2}\right))  \leq -\frac{1}{2}\log \hat \theta(\gamma(0)) - \frac{1}{2} \log \hat \theta(\gamma(1)) - (K-\frac{8nc}{2})\frac{1}{4}d(\gamma(0),\gamma(1))^2.
\end{align*}
\end{remark}
{ In the following we will identify $\hat\theta$ with $\theta$.}


%


\begin{corollary}\label{Cor-Lip-1}
The function $\theta$ is locally Lipschitz {and positive} near any {$p\in \Xreg$.} 
\end{corollary}
\begin{proof}
First observe that semiconcavity of $\theta^{\frac1 N}$, 
the fact that $\theta \ge 0$ and local extendability of geodesics on $\Xreg$
imply that $\theta$ must be locally bounded.
Now, the fact that $\theta$ is Lipschitz near a point $p\in \Xreg$ is a consequence of Proposition~\ref{prop-semiconcave}, the fact that geodesics are locally extendible a definite amount near $p$ by Proposition~\ref{prop:reg-pooints}
and the fact that a semiconcave function on $(0,1)$ is locally Lipschitz.

Let $p\in \Xreg$. Pick any  $q\in \Xreg_*$. Then $\theta(q)>0$. Since a geodesic from $p$ to $q$ can be locally extended past $q$ Proposition~\ref{prop-semiconcave} implies that $\theta(p)>0$. Now the fact that $\theta$ is Lipschitz near $p$ implies that $\theta$ is positive near $p$.

\end{proof}

\begin{cor}\label{dens-limit-exists}
For all $x\in \Xreg$ it holds 
\[
\theta(x)=\lim_{r\to0}\frac{\m(B_{r}(x))}{\omega_{n}r^{n}}.
\]
\end{cor}
\begin{proof}
Pick any $x \in \Xreg$ and let $\Omega\subset \Xreg$ be the set such that for all $y\in\Omega$
\[
\theta(y)=\lim_{r\to0}\frac{\m(B_{r}(y))}{\omega_{n}r^{n}}.
\]
Arguing as in \cite[Lemma 6.1]{kell2017} there is a set of full measure
$\Omega'\subset\Omega$ such that for all $y\in\Omega'$ there is
a unique geodesic $\gamma$ connecting $x$ and $y$ such that $z=\gamma(\frac{1}{2})\in\Omega$.
Note that $\Omega'$ is dense in $\Xreg$.
For $y\in \Omega'$ we can use the arguments of the previous proof and the fact that $z\in\Omega$
to show  
\[
-\log\theta(z)\le-\frac{1}{2}\log\left(\limsup_{r\to0}\frac{\m(B_{r}(x))}{\omega_{n}r^{n}}\right)-\frac{1}{2}\log\theta(y)-(K-\frac{8nc}{2})\frac{1}{4}d(\gamma(0),\gamma(1))^{2}.
\]
Because $\Omega'$ is dense we can take a sequence $y_{n}\to x$ with $y_{n}\in\Omega'$.
Using the fact that $\theta$ is continuous at $x$ we get 
\[
-\frac{1}{2}\log\left(\liminf_{r\to0}\frac{\m(B_{r}(x))}{\omega_{n}r^{n}}\right)\le-\frac{1}{2}\log\left(\limsup_{r\to0}\frac{\m(B_{r}(x))}{\omega_{n}r^{n}}\right)
\]
which implies 
\[
\limsup_{r\to0}\frac{\m(B_{r}(x))}{\omega_{n}r^{n}}\le\liminf_{r\to0}\frac{\m(B_{r}(x))}{\omega_{n}r^{n}}
\]
and thus the claim. 
\end{proof}
\begin{remark}\label{dens-finite-everywhere}
The proof actually shows that $\theta$ is given as a $\lim$ at $x$ whenever $x$ is a continuity point of $\theta$ restricted to $\Xreg\cup\{x\}$. More generally, one obtains  
\[
\limsup_{r\to0}\frac{\m(B_{r}(x))}{\omega_{n}r^{n}}<\infty
\]
 for any $x\in X$.
\end{remark}

The above results allow us to compute the Laplacian in local $\DC$-coordinates using standard Riemannian geometry formulas.

\begin{corollary} Set again $f=\theta$.
Let $V\subset X$ be open and $u\in D({\bf\Delta},V, \m)$.
Then $u\in D({\bf \Delta}_0, \Xreg\cap V, \mathcal{H}^n)$ and 
\begin{equation}\label{another}
{\bf \Delta} u|_{\Xreg\cap V}=  {\bf \Delta}_0 u - \langle \nabla u, \nabla \log f\rangle \m .
\end{equation}
In particular, if $u\in D_{L^2}(\Delta)$, then ${\bf \Delta}_0 u=[ \Delta_0 u] \mathcal H^n$ with $\Delta_0 u\in L^2_{loc}(\Xreg,\m)$.

Let $U=B_\epsilon(p)\subset \Xreg$ be a domain for $\DC$-coordinates and $\tilde{\mathcal A}$ be as in Subsection \ref{ln}. If $u\in \tilde{\mathcal{A}}$ then $ u\in D({\bf \Delta}, \Xreg,\m)$ and $\Delta u|_U\in L^{\infty}_{loc}(\Xreg,\m)$ with 
\begin{align}\label{newid}
{\Delta} u|_{U}= \frac{1}{\sqrt{|g|}}\frac{\partial^{ap}}{\partial x_j}\bigl( g^{jk}\sqrt{|g|}\frac{\partial u}{\partial x_k}\bigr)+ \langle \nabla u, \nabla \log f\rangle.
\end{align}
%
\end{corollary}
\begin{proof}
Since $f$ is locally Lipschitz on $\Xreg$, $u\in D({\bf \Delta}_0, \Xreg\cap V, \mathcal{H}^n)$ follows exactly as in the proof of Proposition 4.18 in \cite{giglistructure}.

Formula \eqref{another} follows from the previous statement since $\Delta u\in L^2(\m)$, $f$ is positive on $\Xreg$ and  $f$ and $\log f$ are locally Lipschitz on $\Xreg$.

For the second claim about $u\in \tilde {\mathcal A}$ first recall that since $(X,d,m)$ is $\RCD$, for any $q\in X$ we have that $d_q$ lies in $D({\bf \Delta},U\backslash \{q\},m)$  and  ${\bf \Delta}d_q$ is locally bounded above on $ U\backslash \{q\}$ by $const\cdot m$ by 
Laplace comparison in $\RCD$ spaces \cite{giglistructure}.

Furthermore, since all geodesics in $U$ are locally extendible  
we have ${\bf \Delta}d_q= \left[{\bf \Delta} d_q\right]_{ac} \cdot\m$ on $U\backslash \left\{q\right\}$ and $\left[{\bf \Delta} d_q\right]_{ac}$ is locally bounded below on $ U\backslash \{q\}$ again by Corollary 4.19 in \cite{cavmon}. Therefore ${\bf \Delta} d_q $ is in $L^\infty_{loc}(U\backslash \{q\},\m)$.

By the chain rule for ${\bf \Delta}$ \cite{giglistructure} the same holds for any $u,h\in \tilde {\mathcal A}$ on all of $U$ as by construction $u$ and $h$ only involve distance functions to points outside $U$.

Finally, since $u\in D({\bf \Delta}, U,\m)$, \eqref{newid} follows from \eqref{another} together with \eqref{equ:laplace}. 
\end{proof}

\section{Continuity of Tangent Cones}\label{sec: t-cones-cont}

Colding and Naber proved in ~\cite{coldingnaberI}  that for  Ricci limits same-scale tangent cones are continuous along interiors of limit geodesics.

We prove that this property holds for $\CD$+$\CAT$ spaces.

\begin{theorem}\label{t-cone-cont}
Let $(X,d,m)$ satisfy \eqref{eq:cd+cat}. Let $\gamma \co [0,1]\to X$ be a geodesic. Let $0<s_0<1/2$.

Then for any $\eps>0$ there is $\delta>0$ {and $r_0>0$ such that for all $0<r<r_0$ and} and for all $t,t'\in [s_0,1-s_0]$ with $|t-t'|<\delta$ it holds that

\[
d_{GH}\left((\frac{1}{r}B_r(\gamma(t)), \gamma(t)), (\frac{1}{r}B_r(\gamma(t')), \gamma(t')) \right)<\eps
\]
\end{theorem}

Note that the theorem implies that same scale tangent cones (if they exists) are uniformly continuous on $ [s_0,1-s_0]$.

The following lemma is well-known (see e.g. ~\cite[Theorem 1.6.15]{BBI}).

\begin{lemma}\label{noncontracting-isom}
Let $K$ be a compact metric space. Let $f\co K\to K$ be distance non-decreasing.
Then $f$ is an isometry.
\end{lemma}

\begin{proof}[Proof of Theorem \ref{t-cone-cont}]

Without loss of generality we can assume that the length of $\gamma$ is less than $\min\{1,\pi_\kappa/2\}$.

We will use the convention that $\kappa(\delta)$ denotes a positive function on $[0,\infty)$ such that $\kappa(\delta)\to 0$ as $\delta\to 0$.

Suppose the theorem is false. Then it fails for some $\eps>0$.  That is there exist $t_i, t_i'\in [s_0, 1-s_0]$ with $|t_i-t_i'|\to 0$ and $r_i\to 0$ such that

\[
d_{GH}\left((\frac{1}{r_i}B_{r_i}(\gamma(t_i)), \gamma(t_i)), (\frac{1}{r_i}B_{r_i}(\gamma(t_i')), \gamma(t_i')) \right)\ge \eps
\]

By passing to a subsequence we can assume that  $t_i, t_i'\to t_0\in [s_0,1-s_0]$. By the triangle inequality and by possibly relabeling and switching $t_i$ with $t_i'$ we can assume that 

\[
d_{GH}\left((\frac{1}{r_i}B_{r_i}(\gamma(t_i)), \gamma(t_i)), (\frac{1}{r_i}B_{r_i}(\gamma(t_0)), \gamma(t_0)) \right)\ge \eps \ \mbox{ 
for all $i$.}
\]

 By passing to a subsequence we can assume that $ (\frac{1}{r_i}X_i,\gamma(t_0))$ pointed GH-converges to some tangent cone $(T_{\gamma(t_0)}X,\cvertex)$ so that

\[
d_{GH}\left((\frac{1}{r_i}B_{r_1}(\cvertex), \cvertex), (\frac{1}{r_i}B_{r_i}(\gamma(t_0)), \gamma(t_0)) \right)\to 0
\]

Let $p=\gamma(0), q=\gamma(1)$. Assume $t_i>t_0$.
Consider the homothety map $\Phi_i$ centered at $p$ such that $\Phi_i(\gamma(t_i))=\gamma(t_0)$ and let $\Psi_i$ be the homothety map centered at $q$ such that $\Psi_i(\gamma(t_0))=\gamma(t_i)$.

Since $L(\gamma)<\pi_\kappa/2$, by the $\CAT(\kappa)$-condition both of these maps are $1$-Lipschitz and by the $\RCD$ condition are almost measure nondecreasing on $B_{10\delta_i}(\gamma(t_0))$ (meaning that the image of any set $A$ has measure $\ge (1-\kappa(\delta_i)m(A)$) where $\delta_i=|t_0-t_i|$. 

Then the composition $f_i=\Phi_i\circ \Psi_i$ maps $B_{r_i}(\gamma(t_0))$ to itself and is measure almost nondecreasing. Together with Bishop-Gromov this implies that
 \[
 \frac {m(B_{r_i}(\gamma(t_i))}{m(B_{r_i}(\gamma(t_0))}\ge 1-\kappa(\delta_i)
 \]
The same argument for $g_i=\Psi_i\circ \Phi_i$ shows that

 \[
  1- \kappa(\delta_i)\le \frac {m(B_{r_i}(\gamma(t_i))}{m(B_{r_i}(\gamma(t_0))}\le  1+ \kappa(\delta_i)
 \]

This in turn implies that  the image of $\Psi_i$ is $\kappa(\delta_i)\cdot r$ dense in $B_{r_i}(\gamma(t_i))$ since it has almost full volume. The same holds for $\Phi_i$ and for $f_i$ for similar reasons.

By Gromov's Arzela\textendash Azcoli theorem the  rescaled maps $f_i\co r_i^{-1}\bar B_{r_i}(\gamma(t_0))\to r_i^{-1}  \bar B_{r_i}(\gamma(t_0))$ subconverge to a self map of the unit ball in the tangent space $f\co \bar B_1(\cvertex)\to \bar B_1(\cvertex)$.

Moreover, $f$ is $1$-Lipschitz and onto and hence its inverse $f^{-1}$ is non-contracting. By Lemma~\ref{noncontracting-isom} $f^{-1}$ is an isometry and hence so is $f$.

Therefore $f_i$ is a $\mu_i$-GH-approximation with $\mu_i\to 0$. Therefore $\Psi_i\co B_{r_i}(\gamma(t_0))\to B_{r_i}(\gamma(t_i))$ is also a  $\mu_i$-GH-approximation. This is a contradiction for large $i$.

\end{proof}
\begin{remark}
Theorem~\ref{t-cone-cont} can be used to give an alternate proof of convexity of $\Xreg$ than the one given in Theorem~\ref{th:reg-points}. Indeed, let $\gamma \co [0,1]\to X$ be a geodesic with $\gamma(0),\gamma(1)\in \Xreg$. Then the set $\gamma^{-1}(\Xreg)$ is nonempty, it is open by Proposition~\ref{prop:reg-pooints} and closed by Theorem~\ref{t-cone-cont}. Therefore $\gamma^{-1}(\Xreg)=[0,1]$.
\end{remark}
Theorem~\ref{t-cone-cont} can be improved to show that the uniform convergence holds with respect to pointed measured Gromov\textendash  Hausdorff convergence. This  was proved for Ricci limits in~\cite{Kap-Li}.

Recall that for doubling metric measure spaces Sturm's $\D$-convergence is equivalent to mGH convergence, see \cite{stugeo1} for details on the transport distance $\D$.

\begin{theorem}\label{t-cone-mGH-cont}
Let $(X,d,m)$ satisfy \eqref{eq:cd+cat}. Let $\gamma \co [0,1]\to X$ be a geodesic. Let $0<s_0<1/2$.

Then for any $\eps>0$ there is $\delta>0$ and $r_0>0$ such that for all $0<r<r_0$ and $t,t'\in [s_0,1-s_0]$ with $|t-t'|<\delta$ it holds that

\[
{\D}\left((\frac{1}{r}\bar B_r(\gamma(t)),\frac{1}{\m(\bar{B}_{r}(\gamma(t)))}{\m}|_{\bar B_{r}(\gamma(t))}), (\frac{1}{r}\bar B_r(\gamma(t')),\frac{1}{\m(\bar{B}_{r}(\gamma(t')))}{\m}|_{\bar B_{r}(\gamma(t'))}) \right)<\eps
\]
\end{theorem}

\begin{proof}{
Suppose the statement is false. Then, for some $\epsilon>0$ one can find $t_i,t_i' \in (s_0,1-s_0)$ with $|t_i-t_i'|\rightarrow 0$ and $r_i\downarrow 0$ such that 
\begin{align*}
{\D}\left((\frac{1}{r_i}\bar{B}_{r_i}(\gamma(t_i)),\frac{1}{\m(\bar{B}_{r_i}(\gamma(t_i)))}\m|_{\bar B_{r_i}(\gamma(t_i))}), ((\frac{1}{r_i}\bar{B}_{r_i}(\gamma(t_i')),\frac{1}{\m(\bar{B}_{r_i}(\gamma(t'_i)))}{\m}|_{\bar B_{r_i}(\gamma(t_i'))})\right)\geq \epsilon
\end{align*}
for all $i\in \N$.  As before we can assume that $t'_i=t_0$ is fixed. To simplify notations we set $\frac{1}{r_i}\bar B_{r_i}(\gamma(t_i))= B_i$ and $\frac{1}{r_i}\bar B_{r_i}(\gamma(t_0))={B}'_i$, and the corresponding probability measure  we denote $ \m_i'$ and $\m_i$, respectively.

We already showed that $B_i$ and $ B_i'$ are GH-close for $i$ large. More precisely, $\Psi_i:B_i\rightarrow  B_i'$ is a $1$-Lipschitz map and a $\mu_i$-GH-approximation with $\mu_i\rightarrow 0$ for $i\rightarrow \infty$. Similar, $\Phi_i: B_i'\rightarrow B_i$ is $1$-Lipschitz map and $\mu_i$-GH-approximations as well. Note that $\Phi_i$ and $\Psi_i$ are indeed $1$-Lipschitz.

Letting $i\rightarrow \infty$ and after taking a subsequence we deduce that $B_i$ and $B_i'$ converge in pointed GH-sense to limit spaces $\bar B_1(\cvertex)$ and $\bar {B}_1'(\cvertex')$, respectively. The set $\bar B_1'(\cvertex)$ is again the (closed) $1$-ball in a tangent space at $\gamma(t_0)$.
 The maps $\Psi_i$, $\Phi_i$ and $f_i=\Phi_i\circ \Psi_i$ converge in Gromov\textendash Arzela\textendash  Ascoli sense to isometries $\Phi:\bar B'_1(\cvertex)\rightarrow \bar{B}_1(\cvertex)$,  $\Psi:{\bar B}_1(\cvertex)\rightarrow \bar B_1'(\cvertex)$ and $f:\bar B_1(\cvertex')\rightarrow \bar B_1(\cvertex)$ where $f=\Phi\circ \Psi$.

Moreover, possibly after taking another subsequence, $\m_i$ and $ \m_i'$ converge to measures $\m_{\infty}$ and ${\m}_{\infty}'$ on $\bar B_1(\cvertex)$ and $\bar{B}_1'(\cvertex')$ repsectively. In particular, $(B_i,\m_i)$ and $({B}_i',\m_i')$ converge in the measured GH-sense,  and hence w.r.t. ${\bf D}$,  to $(\bar B_1(\cvertex),\m_{\infty})$ and $(\bar B_1'(\cvertex'),\m_{\infty}')$ respectively.

{\it Claim:} $(\Psi)_{\#}\m_{\infty}=\m_\infty'$.

First, we show that $f$ is measure preserving: Let
$q_i\in B_i$ converge to $q\in \bar B_1'(\cvertex')$. Then $m_i(B_R(q_i))\to m_\infty(B_R(q))$ for $R>0$ sufficiently small. Also $f_i(q_i)\to f(q)$ and $m_i(B_R(f_i(q_i)))\to m_\infty(B_R(f(q)))$. By the volume noncontracting property of $f_i$ we have that $m_i(f_i(B_R(q_i)))\ge (1-k(\delta_i)) \m_i(B_R(q_i))$. Since $f(B_R(q_i))\subset B_{R}(f_i(q_i))$, it follows $m_i( B_{R}(f_i(q_i))) \ge (1-\kappa(\delta_i))m_i(B_R(q_i))$. Since $m_i( B_{R}(f_i(q_i))) \to m_\infty(B_R(f(q))))$ by passing to the limit we get that $m_\infty(B_R(f(q))\ge m_\infty(B_R(q))$. Since this holds for an arbitrary ball $B_R(q)$, Vitali's covering theorem implies that $f$ is measure nondecreasing. But since $f$ can not increase the overall measure of $\bar B_1(\cvertex)$ this implies that $f$ is measure preserving.

The same argument shows that $\Psi$ and $\Phi$ are measure nondecreasing. 


The combination of the previous two steps yields the claim.

Finally, we obtain that $\Psi: \bar B_1(\cvertex)\rightarrow  \bar B_1'(\cvertex')$ is a metric measure isomorphism and consequently ${\bf D}\left( (\bar B_1(\cvertex),\m_{\infty}), ( \bar B_1'(\cvertex'),\m_{\infty}')\right)=0$. Hence, ${\bf D}\left( (B_i,\m_i),(B_i', \m_i')\right)\rightarrow 0$ for $i\rightarrow \infty$. That is a contradiction.}
\end{proof}

Similarly to ~\cite{Kap-Li} we also obtain that same scale tangent cones along the interior of $\gamma$ have the same dimension.

\begin{theorem}
Let $(X,d,m)$ satisfy \eqref{eq:cd+cat}. Let $\gamma \co [0,1]\to X$ be a geodesic. Let $t,t'\in (0,1)$. Then for any $r_i\to 0$ if there exist  the tangent cones $(T_{\gamma(t)}X,d_t, m_\infty, \cvertex_t)$ and  $(T_{\gamma(t')}X,d_{t'}, m_\infty', \cvertex_{t'})$ corresponding to rescalings $\frac{1}{r_i}$ then $\dim T_{\gamma(t)}X = \dim T_{\gamma(t')}X$.
\end{theorem}
\begin{proof}
Without loss of generality we can assume that the length of  $\gamma$ is $\le \min\{\pi_\kappa/2,1\}$.
Suppose the theorem is false and there exist $t,t'\in (0,1)$ and $r_i\to 0$ such that the corresponding tangent cones have different dimensions.
Let $m=\dim T_{\gamma(t)}X$ and $n=\dim T_{\gamma(t')}X$. Without loss of generality $t<t'$ and $m<n$.
 As before we have a homothety $\Psi$ centered at $q=\gamma(1)$ such that $\Psi(\gamma(t))=\gamma(t')$ and a homothety  $\Phi$ centered at $p=\gamma(0)$ such that $\Phi(\gamma(t'))=\gamma(t)$. By the $\CAT(\kappa)$ condition $\Psi$  is 1-Lipschitz and hence $\Psi(B_{r_i}(\gamma(t))\subset B_{r_i}(\gamma(t'))$. Also by the Brunn\textendash Minkowski inequality $\Psi$ satisfies that

\begin{equation}
m(\Psi(A))\ge c \cdot m(A) \quad \text{ for any }A\subset B_{r_i}(\gamma(s))
\end{equation}
where $1>c=c(K,N, t, t')>0$.
Since we also have a similar inequality of $\Phi$ we get that 

 \[
c \le \frac {m(B_{r_i}(\gamma(t))}{m(B_{r_i}(\gamma(t'))}\le  \frac{1}{c}
 \]
 
 Passing to the limit as in the proof of Theorem~\ref{t-cone-mGH-cont} we get a limit map $\Psi_\infty\co \bar B_1(\cvertex_t)\to \bar B_1(\cvertex_{t'}) $.  Moreover, the map $\Psi_\infty$ is $1$-Lipschitz and satisfies
 \[
m_\infty'(B_r( \Psi_\infty(x))\ge C m_\infty(B_r(x))
 \]
for any $r>0,\ x\in T_{\gamma(t)}X$ and some $C>0$. Let us pick $x\in \Xreg( T_{\gamma(t)}X)$. Then the density $\theta_m(x)=\lim_{r\to0}\frac{\m_\infty(B_{r}(x))}{\omega_{m}r^{m}}$ is defined and positive by Corollary~\ref{dens-limit-exists}. But then for $y=\Psi_\infty(x)$ we also have that $m_\infty'(B_r(y))\ge C_1r^m$ for some $C_1>0$ and all small $r>0$.  On the other hand by Remark~\ref{dens-finite-everywhere} we have that 
 $m_\infty'(B_r(y))\le C_2 r^n$ for some constant $C_2>0$ and all small $r>0$. This is a contradiction since $m<n$.
\end{proof}

\section{Weakly Stably Non-branching $\protect\CAT(1)$ spaces}\label{sec-class-c}

In this section we show that most of the results of Sections  \ref{sec: CD+CAT-str}  and ~\ref{sec:boundary} also apply to 
a more general class of $\CAT(1)$ spaces that satisfy a stable form of the 
non-branching condition. We decided to present this as a separate proof because it 
is somewhat less intuitive as it is done in an inductive way due to the lack of a 
splitting theorem (resp. the suspension theorem). 

Since in this section we will never consider blow up tangent cones  and will only work with geodesic tangent cones we will drop the superindex $g$ when denoting geodesic tangent cones and geodesic spaces of directions. Further, we will only deal with the splitting dimension of $\CAT$ spaces and therefore $\dim $ will denote $\dims$.

Let $\mathcal{NB}$ be the class of non-branching $\CAT(1)$ spaces
and define inductively 
\begin{align*}
\mathcal{C}_{1} & =\{X\in\mathcal{NB}\,|\,\dim X=1\}\\
\mathcal{C}_{n} & =\{X\in\mathcal{NB}\,|\,\dim X\le n,\forall p:T_{p}X\in\mathcal{NB},\Sigma_{p}X\in\mathcal{T}_{n-1}\}
\end{align*}
where the subclass $\mathcal{T}_{n}$ of $\mathcal{C}_{n}$ is defined
as follows: 
\[
\mathcal{T}_{n}=\{X\in\mathcal{C}_{n}\,|\,\diam X\le\pi,\forall v\in X:|\Sph_{\pi}^{X}(v)|\le1\}.
\]
{Here $|A|$ denotes the cardinality of a set $A$.}

It is not difficult to see that the class $\mathcal{T}_n$ contains convex 
balls of the $n$-sphere and $\mathcal{C}_n$ contains all
$n$-dimensional smooth Riemannian manifolds that are globally $\CAT(1)$  and 
whose boundary is either empty or convex and smooth.

For $X\in\CC_n$ we define the geometric boundary $\bCAT X$ in the same way it was defined in section \ref{sec:boundary}. Also as before we set the regular set to be $\mathcal R=\cup_m\mathcal R_m$ where $\mathcal R_m=\{p\in X| T_pX\cong \R^m\}$.}

Let $\mathcal{C}_{n}^{*}$ be the subclass of $\mathcal{C}_{n}$ consisting of at most $n$-dimensional manifolds with boundary such 
that $\partial X = \bCAT X$ and $X\backslash \partial X$ is strongly convex, i.e. if $\gamma$ is a geodesic with $\gamma(t)$ regular
for some $t\in(0,1)$ then $\gamma(t)$ is regular for all $t\in(0,1)$.

The main theorem of this section is the following:
\begin{theorem}
\label{thm:nb-CAT1}The following holds for all $n\in\mathbb{N}$:
\begin{itemize}
\item $\mathcal{C}_{n}=\mathcal{C}_{n}^{*}$;
\item if $X\in\mathcal{T}_{n}$ then either $\rad X<\pi$ and $\partial X\ne\varnothing$
or $X\cong\mathbb{S}^{l}$ where $l=\dim X$;
\item if $X\in\mathcal{T}_{n}$ admits opposites $v,-v\in X$ then either
$X\cong\mathbb{S}^{n}$ or $v,-v\in\bCAT X$. 
\end{itemize}
\end{theorem}

We prove the theorem inductively and start by classifying one-dimensional
spaces. The proof of this elementary lemma is left to the reader.
\begin{lemma}
\label{lem:classification-1dim}A space $(X,d)$ is in $\mathcal{C}_{1}$
if and only if $(X,d)$ is isometric to a circle $\lambda\mathbb{S}^{1}$
with $\lambda\ge1$ or a closed connected subset of $\mathbb{R}^{1}$.
Furthermore, if $(X,d)\in\mathcal{T}_{1}$ then either $X\cong\mathbb{S}^{1}$
or $X$ is an interval of length at most $\pi$. In particular, Theorem
\ref{thm:nb-CAT1} holds for $n=1$.
\end{lemma}
The classes $\mathcal{C}_{n}$ and $\mathcal{T}_{n}$ behave well
w.r.t. geometric constructions.
\begin{lemma}
\label{lem:closed-under-geometric-construction}For all $n\ge1$ the
following holds: 
\begin{itemize}
\item $X\in\mathcal{C}_{n}$ if and only if $X\times\mathbb{R}\in\mathcal{C}_{n+1}$
\item $X\in\mathcal{T}_{n}$ if and only if the Euclidean cone $C(X)$ over $X$ is in $\mathcal{C}_{n+1}$
\item $X\in\mathcal{T}_{n}$ if and only if the spherical suspension $S(X)$
over $X$ is in $\mathcal{T}_{n+1}$.
\end{itemize}
\end{lemma}
\begin{proof}
For $n=1$ the claim follows from Lemma \ref{lem:classification-1dim}. 

Assume $n>1$ and the three claims hold for $n$: If $X\in\mathcal{C}_{n+1}$
and $Y=X\times\mathbb{R}$ then $T_{(p,r)}Y\cong T_{p}X\times\mathbb{R}$
is non-branching and $\Sigma_{(p,r)}Y\cong S(\Sigma_{p}X)$.
Since $\Sigma_{p}X\in\mathcal{T}_{n-1}$ we must have $\Sigma_{(p,r)}Y\in\mathcal{T}_{n}$. 
Similarly, if $Y\in\mathcal{C}_{n+2}$ then $T_{(p,r)}Y\cong T_{p}X\times\mathbb{R}$
is non-branching and $\Sigma_{(p,r)}Y\cong S(\Sigma_{p}X)\in\mathcal{T}_{n+1}$.
The induction step implies $\Sigma_{p}X\in\mathcal{T}_{n}$. Because
$p\in X$ was arbitrary we obtain the claim.

For the second claim note $T_{\cvertex}{ C(X)}\cong C(X)$ and $T_{(p,r)}{ C(X)}\cong T_{p}X\times\mathbb{R}$
for $r>0$ and $p\in X$. By definition of $\mathcal{C}_{n+1}$ we
see $C(X)\in\mathcal{C}_{n+2}$ implies $X\in\mathcal{T}_{n+1}$.
If $X\in\mathcal{T}_{n+1}$ then one readily verifies that all geodesic
tangent cones are non-branching and $\Sigma_{\cvertex}{ C(X)}=X\in\mathcal{T}_{n+1}$.
By the  induction assumption and the fact that $\Sigma_{p}X\in\mathcal{T}_{n}$
we see $S(\Sigma_{p}X)\cong\Sigma_{(p,r)}{ C(X)}\in\mathcal{T}_{n+1}$ implying
$C(X)\in\mathcal{C}_{n+2}$. 

Let us prove the third claim.  Suppose $S(X)\in \TT_{n+1}$.  Since  $X\cong \Sigma_{\cvertex}(SX)$ and  $\TT_{n+1}\subset  \CC_{n+1}$ by the definition of $ \CC_{n+1}$ this implies that $X\in \TT_n$.

Conversely, suppose $X\in \TT_n$. The conditions $\diam X\le \pi$ and  $|\Sph_{\pi}^{X}({p})|\le1$ for all $p\in X$  easily imply that the same holds for $S(X)$. As before it's easy to see that $S(X)$ is non-branching. 

Next,  for $0<r<\pi$ and $p\in X$ we have that $\Sigma X_{(p,r)}\cong S(\Sigma_p X)\in \TT_n$ by the induction assumption. Also $\Sigma_{\cvertex}(S(X))\cong X\in \TT_n$. Hence $S(X)\in\CC_{n+1}$. This finishes the proof of the third claim and of the lemma.
\end{proof}
\begin{corollary}
\label{cor:good-tangent-cone}If $X\in\mathcal{C}_{n}$ then $T_{p}X\in\mathcal{C}_{n}$
for all $p\in X$. 
\end{corollary}
\begin{corollary}
A subclass $\mathcal{D}_{n}$ of at most $n$-dimensional non-branching
$\CAT(1)$ spaces is in $\mathcal{C}_{n}$ if for all $X\in\mathcal{D}_{n}$ and 
any $p\in X$ the geodesic tangent cone $T_{p}X$ is in $\mathcal{D}_{n}$. In particular, the class of $\CAT(1)$ spaces satisfying
the $\mathrm{MCP}(K,N)$ condition with $K\le 0$  is in $\mathcal{C}_{\lfloor N\rfloor}$.
\end{corollary}
\begin{proof}
{ Observe that for any $\CAT(1)$ space $X$ and  $v\in\Sigma_{p}X$
and $t>0$ it holds
\[
T_{(v,t)}(T_{p}X)\cong T_{v}(\Sigma_{p}X)\times\mathbb{R}
\]
which is non-branching if and only if $T_{v}(\Sigma_{p}X)$ is
non-branching. 

We can now prove the Corollary by induction on $n$. The base of induction $n=1$ is easy and is left to the reader.
Now suppose the statement holds for $n-1\ge 1$ and a class $\mathcal{D}_{n}$ satisfies the assumptions of the lemma. Let $\mathcal{D}_{n-1}$ be the class of non-branching $\CAT(1)$ spaces such that for every $Y\in \mathcal{D}_{n-1}$ and every $q\in Y$ it holds that $T_qY\times \R\in \mathcal D_n$. Then $\mathcal{D}_{n-1}$ satisfies the induction assumptions for $n-1$ and hence $\mathcal{D}_{n-1}\subset \CC_{n-1}$. In particular $\Sigma_pX\in \mathcal D_{n-1}\subset \CC_{n-1}$ for any $X\in \mathcal D_n, p\in X$. Since $T_pX=C(\Sigma_pX)$ is non-branching this implies that in fact $\Sigma_pX\in  \mathcal T_{n-1}$ and hence $X\in \mathcal C_n$ by the definition of $\CC_n$.}

The last claim follows by observing that the class of $\CAT(1)$ spaces with 
$\MCP(K,N)$ condition for some $K \in \mathbb{R}$ is stable under taking 
GH-tangent spaces. As in Remark \ref{geod-cone-rcd} this shows that 
geodesic tangents are $\CAT(0)$ spaces with $\MCP(0,N)$ condition as well. 
\end{proof}
\begin{remark}
The corollary also implies that $\mathcal{C}_{n}$ agrees with the class of 
$n$-dimensional non-branching $\CAT(1)$ spaces that is 
stable under taking geodesic tangents.
\end{remark}

We continue with the proof of Theorem \ref{thm:nb-CAT1} with the following 
well-known fact about $\CAT(1)$ spaces.

\begin{lemma}
If $X$ is a $\CAT(1)$ space with $\rad X<\pi$ then $X$ is contractible.
\end{lemma}
%

The following lemma is a replacement of Proposition \ref{inner-iff-non-contr-iff-non-contr-minus-geod}
as the space of directions of spaces in $\mathcal{C}_n$ might not satisfy a suspension theorem. 

\begin{lemma}\label{c-int-geod}
\label{lem:mfl-ind}Let $n>1$ and assume Theorem \ref{thm:nb-CAT1} holds
for $n-1$. Then for all $X\in\mathcal{C}_{n}$ and all non-trivial
geodesics $\gamma:[0,1]\to X$ the following are equivalent:
\begin{enumerate}
\item $\Sigma_{\gamma(t)}X$ is non-contractible for some/all $t\in(0,1)$
\item $\Sigma_{\gamma(t)}X\backslash\{\pm\dot{\gamma}(t)\}$ is non-contractible
for some/all $t\in(0,1)$
\item $\Sigma_{\gamma(t)}X\cong\mathbb{S}^{l}$ where $l=\dim X\le n$ for
some/all $t\in(0,1).$
\end{enumerate}
\end{lemma}
\begin{proof}
Note first that by Lemma \ref{lem-kramer} the equivalence holds for
all $t\in(0,1)$ if it holds for some $t\in(0,1)$. 

For $n=2$ we know $\Sigma_{\gamma(t)}X$ is contractible for some
$t\in(0,1)$ if and only if $\Sigma_{\gamma(t)}X\backslash\{\pm\dot{\gamma}(t)\}$
has one contractible component if and only if $\Sigma_{\gamma(t)}X\cong[0,\pi]$. 

Assume $n>2$ and $Y=\Sigma_{\gamma(t)}X$ is contractible for some
$t\in(0,1)$. Then $\rad Y<\pi$ and $\pm\dot{\gamma}(t)\in\partial Y$
by the statement of Theorem \ref{thm:nb-CAT1} for $n-1$. Thus there
is a regular point $v\in Y$ and $r\in[\rad Y,\pi)$ with 
\[
\bar{B}_{r}^{Y}(v)=Y.
\]
Since $v$ is regular and $\pm\dot{\gamma}(t)$ are boundary points, 
all geodesics from $w\in\bar{B}_{r}^{Y}(v)\backslash\{\pm\dot{\gamma}(t)\}$
to $v$ avoid $\pm\dot{\gamma}(t)$. In particular, the geodesic contraction
towards $v$ induces a contraction of $\bar{B}_{r}^{Y}(v)\backslash\{\pm\dot{\gamma}(t)\}$
onto $\{w\}$. Hence $\bar{B}_{r}^{Y}(v)\backslash\{\pm\dot{\gamma}(t)\}$
is contractible. 

If on the other hand $Y\backslash\{\pm\dot{\gamma}(t)\}$ was contractible
then $\rad Y<\pi$ so that $Y$ must be contractible and not isometric
to $\mathbb{S}^{l}$. Finally, if $Y$ is not isometric to $\mathbb{S}^{l}$
then again $\rad Y<\pi$ so that $Y$ and $Y\backslash\{\pm\dot{\gamma}(t)\}$
are both contractible. 
\end{proof}
\begin{corollary}\label{cor:Cn-Cnstar}
If $n>1$ and Theorem \ref{thm:nb-CAT1} holds for $n-1$ then $\mathcal{C}_{n}=\mathcal{C}_{n}^{*}$ 
\end{corollary}
\begin{proof}
 Let $X\in\mathcal{C}_{n}$.  By the same argument as in the proof of Theorem~\ref{c-dim} the set of regular points is non-empty  and agrees with the set of points having non-contractible 
spaces of directions. 
The same argument as in the proof of Proposition~\ref{prop:reg-pooints} shows that the regular set is open and is an $n$-manifold without boundary. 

The same argument as in the proof of Proposition~\ref{th:reg-points} but using Lemma~\ref{c-int-geod} instead of Proposition~\ref{inner-iff-non-contr-iff-non-contr-minus-geod} shows that the regular set is strongly convex and dense.

Next, observe that in the proof of  Theorem \ref{c-manifold-with-boundary} the analysis of the topology of $X$ near boundary points 
only relies on local uniqueness of geodesics and the above properties the regular set.

Lastly, let us verify that  $\bCAT X=\partial X$.

Note that the space of directions $\Sigma_p X$ of
 a point $p \in \partial X$ must be contractible and, in particular, not isometric to a sphere. As $\Sigma_p X \in \mathcal{T}_{n-1}$ Theorem \ref{thm:nb-CAT1} shows $\bCAT \Sigma_p X = \partial \Sigma_p X \ne \varnothing$. Thus $\bCAT X \subset \partial X$.  Because regular points have neighborhoods homeomorphic to Euclidean balls the opposite inclusion
 is also true. In particular, $\partial X = \bCAT X=\X\backslash \reg$ implying $\mathcal{C}_n =\mathcal{C}_n^*$.
\end{proof}
\begin{lemma}
Suppose Theorem \ref{thm:nb-CAT1} holds for $n-1$. Let $X\in\text{\ensuremath{\mathcal{C}}}_{n}$
and $\rad X<\pi$.  Then $\bCAT X\ne\varnothing$.
\end{lemma}
\begin{proof}
Assume by contradiction that $\bCAT X=\varnothing$ for some $X\in\mathcal{C}_{n}$.
Pick a regular point $x$ and a point $y\in X$ with 
\[
r=d(x,y)=\sup_{y'\in Y}(x,y')\in[\rad X,\pi).
\]
By definition of $r$ we must have $\bar{B}_{r}(x)=X$. 

Assume by contradiction $\bCAT X=\varnothing$. Then $y$ is regular
so that the unit speed geodesic $\gamma:[0,r]\to X$ connecting $x$
and $y$ can be extended to a local unit speed geodesic $\tilde{\gamma}:[0,r+\epsilon]\to M$
for some $\epsilon\in[0,\pi-\rad X]$. However, $r+\epsilon\le\pi$
so that $\tilde{\gamma}$ is still minimal. 
But then $\tilde{\gamma}(r+\epsilon)\notin\bar{B}_{r}(x)=X$. This a 
contradiction and hence $\bCAT X\ne\varnothing$.
\end{proof}
\begin{lemma}
\label{lem:regular-opposites}Suppose Theorem \ref{thm:nb-CAT1} holds
for $n-1$. If $X\in\mathcal{T}_{n}$ admits two opposites $\pm v\in X\backslash\bCAT X$
then $X$ is isometric to $\mathbb{S}^{l}$ for $l=\dim X$.
\end{lemma}
\begin{proof}
For $s\in(0,\pi)$ let $\Sigma_{v}^{s}X$ be the subset of directions
$u\in\Sigma_{v}X$ such that for some $a_u \in (0,\pi-s]$ there is a 
geodesic $\gamma_{u}:[0,a_{u}]\to X$
with $\dot{\gamma}(0)=u$ and $\gamma(a_{u})\in\Sph_{s}(-v)$.
Then the assignment $u\mapsto \gamma(a_{u})$ from  $\Sigma_{v}^{s}X$ to $\Sph_{s}(-v)$
is onto and injective since 
\[
R_{s}=\sup_{w\in\Sph_{s}(-v)}d(v,w)<\pi.
\]
Thus $\Sigma_{v}^{s}X$ is homeomorphic to the $l$-dimensional 
closed manifold $\Sph_{s}(-v)$ where $l=\dim X$. But then 
$\Sigma_{v}^{s}X$ is an $l$-dimensional closed
submanifold of the $l$-dimension closed connected manifold $\Sigma_{v}X$.
Hence $\Sigma_{v}^{s}X=\Sigma_{v}X$. 

We claim $\bCAT X=\varnothing$. Indeed, if this was wrong then
\[
\sup_{y\in\bCAT X}d(v,y)<\pi-\epsilon
\]
{for some small $\eps>0$ 
because $\partial X\subset X\backslash B_{\delta}(\pm v)$ for
all small $\delta>0$ and $v$ has a unique opposite equal to $-v$. Since $X\in\mathcal{C}_{n}=\mathcal{C}_{n}^{*}$ and $\Sigma_{v}^{\eps}X=\Sigma_{v}X$,}
all geodesics starting at $v$ will stay in the regular
set until they hit the subset $\Sph_{\epsilon}(-v)$ of regular
point after a length $a\in[\pi-\epsilon,R_{\epsilon}]$. But then
$d(v,y)>\pi-\epsilon$ for some $y\in\bCAT X$ which is a contradiction. 

Therefore $\partial X=\varnothing$ and hence $X$ is geodesically complete and every geodesic can be extended
to a minimal geodesic of maximal length $\pi$. From $\{-v\}=\Sph_{\pi}(v)$
we see that   $d(x,v)+d(x,-v)=\pi$ for any $x\in X$. Thus by Lemma~\ref{cat-susp} 
$X$ is the spherical suspension $S(Y)$ over the convex set $Y=\{y\in X\,|\,d(v,x)=d(x,-y)=\frac{\pi}{2}\}\cong \Sigma_vX$.
Because $Y\in\mathcal{C}_{n-1}$ by Lemma \ref{lem:closed-under-geometric-construction}
and $\bCAT Y=\varnothing$, Theorem \ref{thm:nb-CAT1} for $n-1$
implies the claim. 
\end{proof}
\begin{corollary}
Suppose Theorem \ref{thm:nb-CAT1} holds for $n-1$. Whenever $X\in\mathcal{T}_{n}$
with $\rad X=\pi$ then $X$ is isometric to $\mathbb{S}^{l}$ for
$l=\dim X$.
\end{corollary}
\begin{proof}
By assumption $\diam X=\pi$ and for all $x\in X$
\[
\sup_{y\in X}d(x,y)=\pi.
\]
In other words every point $x\in X$ has a (necessarily unique) opposite $-x$.
We claim that there is a regular $v$ such that $-v$ is regular as well.

Choose a regular point $x\in X$. 
Since $x$
is regular for some small $\eps>0$ there is a unit speed local geodesic $\gamma:[-2\epsilon,\pi]\to X$
such that $\gamma\big|_{[0,\pi]}$ connects $x$ and $-x$. 

Since $x=\gamma(0)$ is regular and $X\in\mathcal{C}_{n}^{*}=\mathcal{C}_{n}$, the points $\gamma(-\epsilon)$
and $\gamma(\pi-\epsilon)$ are regular as well. Because local geodesics of length at most $\pi$ in $\CAT(1)$ spaces are geodesics
we know that $\gamma\big|_{[-\epsilon,\pi-\epsilon]}$ is minimal. Thus $d(\gamma(-\epsilon),\gamma(\pi-\epsilon))=\pi$
so that $X$ admits regular opposites $v=\gamma(-\epsilon)$ and $-v=\gamma(\pi-\epsilon)$. Now Lemma \ref{lem:regular-opposites} implies 
the claim of the corollary.
\end{proof}
\begin{corollary}\label{cor:direction-dichotomy}
Suppose Theorem \ref{thm:nb-CAT1} holds for $n-1$ and $X\in\mathcal{T}_{n}$
admits opposites $v,-v\in X$. Then either $X$ is isometric to $\mathbb{S}^{l}$
for $l=\dim X$ or $v,-v\in\bCAT X$. 
\end{corollary}
\begin{proof}
Assume $X$ admits opposites $v,-v\in X$ such that $v$ is regular.
By a similar argument as above we see that there is a regular $w$
which is on the local geodesic form by extending the geodesic connecting
$-v$ and $v$ such that $w$ admits an opposite $-w$ that is also
regular. But then $X$ is isometric to a $\SS^l$ where $l=\dim X$. This proves
the claim. 
\end{proof}

The proof of Theorem \ref{thm:nb-CAT1} now follows by induction from Lemma \ref{lem:classification-1dim} and 
Lemma \ref{lem:regular-opposites} and its corollaries.

\small{
\bibliographystyle{amsalpha}

\begin{thebibliography}{DPMR17}

\bibitem[AB18]{ambrosiobertrand}
Luigi Ambrosio and J\'{e}r\^{o}me Bertrand, \emph{D{C} calculus}, Math. Z.
  \textbf{288} (2018), no.~3-4, 1037--1080. \MR{3778989}

\bibitem[AGS13]{agslipschitz}
Luigi Ambrosio, Nicola Gigli, and Giuseppe Savar{\'e}, \emph{Density of
  {L}ipschitz functions and equivalence of weak gradients in metric measure
  spaces}, Rev. Mat. Iberoam. \textbf{29} (2013), no.~3, 969--996. \MR{3090143}

\bibitem[AGS14a]{agsheat}
\bysame, \emph{Calculus and heat flow in metric measure spaces and applications
  to spaces with {R}icci bounds from below}, Invent. Math. \textbf{195} (2014),
  no.~2, 289--391. \MR{3152751}

\bibitem[AGS14b]{agsriemannian}
\bysame, \emph{Metric measure spaces with {R}iemannian {R}icci curvature
  bounded from below}, Duke Math. J. \textbf{163} (2014), no.~7, 1405--1490.
  \MR{3205729}

\bibitem[AHPT18]{AHPT18}
Luigi {Ambrosio}, Shouhei {Honda}, Jacobus~W. {Portegies}, and David
  {Tewodrose}, \emph{{Embedding of $RCD^*(K,N)$ spaces in $L^2$ via
  eigenfunctions}}, arXiv e-prints (2018), arXiv:1812.03712.

\bibitem[BB99]{Ballmann-Brin}
Werner Ballmann and Michael Brin, \emph{Diameter rigidity of spherical
  polyhedra}, Duke Math. J. \textbf{97} (1999), no.~2, 235--259. \MR{1682245}

\bibitem[BBI01]{BBI}
Dmitri Burago, Yuri Burago, and Sergei Ivanov, \emph{A course in metric
  geometry}, Graduate Studies in Mathematics, vol.~33, American Mathematical
  Society, Providence, RI, 2001. \MR{1835418 (2002e:53053)}

\bibitem[Ber02]{berestovskii02}
Valeri{\u \i}~N. Berestovski\u{\i}, \emph{Busemann spaces with upper-bounded
  {A}leksandrov curvature}, Algebra i Analiz \textbf{14} (2002), no.~5, 3--18,
  translation in \textit{St. Petersburg Math. J.} \textbf{14} (2003), no. 5,
  713--723. \MR{1970330}

\bibitem[BH99]{BH99}
Martin~R. Bridson and Andr{\'e} Haefliger, \emph{Metric spaces of non-positive
  curvature}, Grundlehren der Mathematischen Wissenschaften [Fundamental
  Principles of Mathematical Sciences], vol. 319, Springer-Verlag, Berlin,
  1999. \MR{MR1744486 (2000k:53038)}

\bibitem[BN93]{nikolaev}
Valeri{\u \i}~N. Berestovski\u{\i} and Igor~G. Nikolaev, \emph{Multidimensional
  generalized {R}iemannian spaces}, Geometry, {IV}, Encyclopaedia Math. Sci.,
  vol.~70, Springer, Berlin, 1993, pp.~165--243, 245--250. \MR{1263965}

\bibitem[BS18]{brusem}
Elia {Bru{\`e}} and Daniele {Semola}, \emph{{Constancy of the dimension for
  RCD(K,N) spaces via regularity of Lagrangian flows}}, arXiv e-prints (2018),
  arXiv:1804.07128.

\bibitem[Che99]{cheegerlipschitz}
Jeff Cheeger, \emph{Differentiability of {L}ipschitz functions on metric
  measure spaces}, Geom. Funct. Anal. \textbf{9} (1999), no.~3, 428--517.
  \MR{1708448 (2000g:53043)}

\bibitem[CM17]{cavmon}
Fabio Cavalletti and Andrea Mondino, \emph{Sharp and rigid isoperimetric
  inequalities in metric-measure spaces with lower {R}icci curvature bounds},
  Invent. Math. \textbf{208} (2017), no.~3, 803--849. \MR{3648975}

\bibitem[CN12]{coldingnaberI}
Tobias~Holck Colding and Aaron Naber, \emph{Sharp {H}\"older continuity of
  tangent cones for spaces with a lower {R}icci curvature bound and
  applications}, Ann. of Math. (2) \textbf{176} (2012), no.~2, 1173--1229.
  \MR{2950772}

\bibitem[DPG18]{GP-noncol}
Guido De~Philippis and Nicola Gigli, \emph{Non-collapsed spaces with {R}icci
  curvature bounded from below}, J. \'{E}c. polytech. Math. \textbf{5} (2018),
  613--650. \MR{3852263}

\bibitem[DPMR17]{DeP-Mar-Rin}
Guido De~Philippis, Andrea Marchese, and Filip Rindler, \emph{On a conjecture
  of {C}heeger}, Measure theory in non-smooth spaces, Partial Differ. Equ.
  Meas. Theory, De Gruyter Open, Warsaw, 2017, pp.~145--155. \MR{3701738}

\bibitem[EG15]{Gar-Evans}
Lawrence~C. Evans and Ronald~F. Gariepy, \emph{Measure theory and fine
  properties of functions}, revised ed., Textbooks in Mathematics, CRC Press,
  Boca Raton, FL, 2015. \MR{3409135}

\bibitem[Gig15]{giglistructure}
Nicola Gigli, \emph{On the differential structure of metric measure spaces and
  applications}, Mem. Amer. Math. Soc. \textbf{236} (2015), no.~1113, vi+91.
  \MR{3381131}

\bibitem[GMS15]{gmsstability}
Nicola Gigli, Andrea Mondino, and Giuseppe Savar{\'e}, \emph{Convergence of
  pointed non-compact metric measure spaces and stability of {R}icci curvature
  bounds and heat flows}, Proc. Lond. Math. Soc. (3) \textbf{111} (2015),
  no.~5, 1071--1129. \MR{3477230}

\bibitem[GP16]{gipa}
Nicola {Gigli} and Enrico {Pasqualetto}, \emph{{Behaviour of the reference
  measure on $RCD$ spaces under charts}}, arXiv e-prints (2016),
  arXiv:1607.05188.

\bibitem[{Han}19]{Han19}
Bang-Xian {Han}, \emph{{Measure rigidity of synthetic lower Ricci curvature
  bound on Riemannian manifolds}}, arXiv e-prints (2019), arXiv:1902.00942.

\bibitem[{Hon}19]{Honda19}
Shouhei {Honda}, \emph{{New differential operator and non-collapsed \$RCD\$
  spaces}}, arXiv e-prints (2019), arXiv:1905.00123.

\bibitem[Kel17]{kell2017}
Martin Kell, \emph{Transport maps, non-branching sets of geodesics and measure
  rigidity}, Adv. Math. \textbf{320} (2017), 520--573. \MR{MR3709114}

\bibitem[Ket15]{Ket-cones}
Christian Ketterer, \emph{Cones over metric measure spaces and the maximal
  diameter theorem}, J. Math. Pures Appl. (9) \textbf{103} (2015), no.~5,
  1228--1275. \MR{3333056}

\bibitem[Kit17]{kbg}
Yu~Kitabeppu, \emph{A {B}ishop-type inequality on metric measure spaces with
  {R}icci curvature bounded below}, Proc. Amer. Math. Soc. \textbf{145} (2017),
  no.~7, 3137--3151. \MR{3637960}

\bibitem[Kit19]{KitaPOTA}
\bysame, \emph{A {S}ufficient {C}ondition to a {R}egular {S}et {B}eing of
  {P}ositive {M}easure on {S}paces}, Potential Anal. \textbf{51} (2019), no.~2,
  179--196. \MR{3983504}

\bibitem[KK17]{Kap-Ket-18}
Vitali Kapovitch and Christian Ketterer, \emph{{CD} meets {CAT}}, arXiv
  e-prints (2017), arXiv:1712.02839.

\bibitem[KK19]{KaKe19}
\bysame, \emph{{Weakly noncollapsed RCD spaces with upper curvature bounds}},
  arXiv e-prints (2019), arXiv:1901.06966.

\bibitem[KL18]{Kap-Li}
Vitali Kapovitch and Nan Li, \emph{On dimensions of tangent cones in limit
  spaces with lower {R}icci curvature bounds}, J. Reine Angew. Math.
  \textbf{742} (2018), 263--280. \MR{3849628}

\bibitem[Kle99]{kleiner}
Bruce Kleiner, \emph{The local structure of length spaces with curvature
  bounded above}, Math. Z. \textbf{231} (1999), no.~3, 409--456. \MR{1704987}

\bibitem[KM18]{kellmondino}
Martin Kell and Andrea Mondino, \emph{On the volume measure of non-smooth
  spaces with {R}icci curvature bounded below}, 2018. \MR{3801291}

\bibitem[KM19]{KapMon}
Vitali Kapovitch and Andrea Mondino, \emph{On the topology and the boundary of
  {N}-dimensional {RCD(K,N)} spaces}, arXiv:1907.02614, 2019.

\bibitem[Kra11]{Kramer11}
Linus Kramer, \emph{On the local structure and the homology of {${\rm
  CAT}(\kappa)$} spaces and {E}uclidean buildings}, Adv. Geom. \textbf{11}
  (2011), no.~2, 347--369. \MR{2795430}

\bibitem[LN19]{Lytchak-Nagano18}
Alexander Lytchak and Koichi Nagano, \emph{Geodesically complete spaces with an
  upper curvature bound}, Geom. Funct. Anal. \textbf{29} (2019), no.~1,
  295--342. \MR{3925112}

\bibitem[LS07]{Lyt-Schr07}
Alexander Lytchak and Viktor Schroeder, \emph{Affine functions on {${\rm
  CAT}(\kappa)$}-spaces}, Math. Z. \textbf{255} (2007), no.~2, 231--244.
  \MR{2262730}

\bibitem[LV09]{lottvillani}
John Lott and C{\'e}dric Villani, \emph{Ricci curvature for metric-measure
  spaces via optimal transport}, Ann. of Math. (2) \textbf{169} (2009), no.~3,
  903--991. \MR{2480619 (2010i:53068)}

\bibitem[MGPS18]{DGPS18}
Simone~Di Marino, Nicola Gigli, Enrico Pasqualetto, and Elefterios Soultanis,
  \emph{Infinitesimal hilbertianity of locally {$CAT(\kappa)$}-spaces}, arXiv
  e-prints (2018), arXiv:1812.02086.

\bibitem[MN19]{mondinonaber}
Andrea Mondino and Aaron Naber, \emph{Structure theory of metric measure spaces
  with lower {R}icci curvature bounds}, J. Eur. Math. Soc. (JEMS) \textbf{21}
  (2019), no.~6, 1809--1854. \MR{3945743}

\bibitem[Oht07]{ohtamcp}
Shin-ichi Ohta, \emph{On the measure contraction property of metric measure
  spaces}, Comment. Math. Helv. \textbf{82} (2007), no.~3, 805--828.
  \MR{2341840 (2008j:53075)}

\bibitem[Per95]{Per-DC}
G.~Perelman, \emph{{DC} structure on {A}lexandrov space with curvature bounded
  below.}, preprint, http://www.math.psu.edu/petrunin/papers/papers.html, 1995.

\bibitem[Pet11]{palvs}
Anton Petrunin, \emph{Alexandrov meets {L}ott-{V}illani-{S}turm}, M\"unster J.
  Math. \textbf{4} (2011), 53--64. \MR{2869253 (2012m:53087)}

\bibitem[Stu06a]{stugeo1}
Karl-Theodor Sturm, \emph{On the geometry of metric measure spaces. {I}}, Acta
  Math. \textbf{196} (2006), no.~1, 65--131. \MR{2237206 (2007k:53051a)}

\bibitem[Stu06b]{stugeo2}
\bysame, \emph{On the geometry of metric measure spaces. {II}}, Acta Math.
  \textbf{196} (2006), no.~1, 133--177. \MR{2237207 (2007k:53051b)}

\end{thebibliography}

\providecommand{\bysame}{\leavevmode\hbox to3em{\hrulefill}\thinspace}
\providecommand{\MR}{\relax\ifhmode\unskip\space\fi MR }
\providecommand{\MRhref}[2]{%
  \href{http://www.ams.org/mathscinet-getitem?mr=#1}{#2}
}
\providecommand{\href}[2]{#2}

}
\end{document}